\documentclass[11pt,review]{elsarticle}               
\usepackage{graphicx}
\usepackage{amssymb,amsmath,graphicx}
\usepackage{bm}
\usepackage[boxed,commentsnumbered]{algorithm2e}
\usepackage{epstopdf}

 \usepackage{mathptmx}      
\newtheorem{thm}{Theorem}

\newtheorem{cor}[thm]{Corollary}
\newdefinition{rmk}{Remark}
\newproof{proof}{Proof}
 
\renewcommand{\Re}{\mathop{\rm Re}} 
 
\newcommand{\field}[1]{\mathbb{#1}}
\newcommand{\C}{\field{C}} 
\newcommand{\N}{\field{N}}
\newcommand{\R}{\field{R}}  

\newcommand{\Ord}[1]{\mathcal{O}\left( #1 \right)}

\begin{document}  
\begin{frontmatter}
 
\title{Computation of 2D Fourier transforms and diffraction integrals using Gaussian radial basis functions}



\author[address1,address2]{A. Mart\'{\i}nez--Finkelshtein\corref{cor1}}
\ead{andrei@ual.es}
\cortext[cor1]{Corresponding author.}
\author[address1]{D. Ramos-L\'opez}
\ead{drl012@ual.es}
\author[address3]{D.~R.~Iskander}
\ead{robert.iskander@pwr.edu.pl}
\address[address1]{Department of Mathematics,
University of Almer\'{\i}a, Spain}
\address[address2]{Instituto Carlos I de F\'{\i}sica Te\'{o}rica y Computacional,
Granada University, Spain}
\address[address3]{Department of Biomedical Engineering, Wroclaw University of Technology, Wroclaw, Poland}


\begin{abstract}
We implement an efficient method of computation of two dimensional Fourier-type integrals based on approximation of the integrand by Gaussian radial basis functions, which constitute a standard tool in approximation theory. As a result, we obtain a rapidly converging series expansion for the integrals, allowing for their accurate calculation. We apply this idea to the evaluation of diffraction integrals, used for the computation of the through-focus characteristics of an optical system. We implement this method and compare its performance in terms of complexity, accuracy and execution time with several alternative approaches, especially with the extended Nijboer-Zernike theory, which is also outlined in the text for the reader's convenience. The proposed method yields a reliable and fast scheme for simultaneous evaluation of such kind of integrals for several values of the defocus parameter, as required in the characterization of the  through-focus optics. 
\end{abstract}


\end{frontmatter}

\section{Introduction}

The importance of the 2D Fourier transform in mathematical imaging and vision is difficult to overestimate. For a function $g$ given  on $\R^2$ in polar coordinates $(\rho, \theta)$ it reduces to calculating the integrals of the form  
\begin{equation}\label{fourier1}
F(r, \phi)= \frac{1}{\pi} \int_{0}^\infty\int_0^{2\pi} g(\rho  ,   \theta) e^{2\pi i r \rho \cos (\theta- \phi) } \rho d\rho d\theta, 
\end{equation} 
in which the Fourier transform is expressed in polar coordinates $(r,\phi)$. For instance, the impulse response of an optical system (referred to as the point-spread function (PSF)) can be defined in terms of diffraction integrals of the form~\eqref{fourier1}. The PSF uniquely defines a linear optical system and it is usually calculated for a single value of the focus parameter~\cite{Goodman_96}.

Calculating through-focus characteristics of an optical system has a wide variety of important applications including phase-diverse phase retrieval~\cite{Dean_Bowers03,Dean_etal06}, wavefront sensing~\cite{Thurman11}, aberration retrieval in lithography, microscopy, and extreme ultraviolet light optics~\cite{Dirksen_etal04,Dirksen_etal05}, as well as in physiological optics, where such calculation can be used to assess the efficacy of intraocular lenses~\cite{Artal_etal95,Marcos_etal2005,Piers_etal2007}, study the depth-of-focus of the human eye~\cite{Atchison2012,Yi_etal2010} or determine optimal pupil size in retinal imaging instruments such as the confocal scanning laser ophthalmoscope~\cite{Donnelly_Roorda03}. 

Another rapidly developing field in imaging is the digital holography~\cite{2,3,Schnars,1}, which finds applications in the quantitative visualization of phase objects such as living cells using microscopic objectives in a digital holographic set-up. This technology of acquiring and processing holographic measurement data usually is comprised of two steps, the recording of an interference pattern produced by a real object on a CCD, followed by the numerical reconstruction of the hologram by simulating the back-propagation of the image. Fourier-type integrals are an essential part of several algorithms used in digital holography, for instance when calculating the Fresnel Transforms. 

In this context, the numerical algorithms used to calculate the 2D  Fourier transform play a central role. 
Unfortunately, such integrals can be explicitly evaluated in terms of standard special functions only in a limited number of situations. On the other hand,  purely numerical procedures for these evaluations are computationally cumbersome and require substantial computational resources. Thus, the most popular alternative are the semi-analytic methods that use an approximation of the integrand with functions from a certain class, for which some forms of closed expressions are available. This is the paradigm of the extended Nijboer--Zernike (ENZ) approach, explained briefly in Section~\ref{sec:ENZ}.
The ENZ theory constitutes an important step forward in comparison to the direct quadrature integration. Its initial implementation \cite{Braat_etal2002,Janssen2002,Janssen2004} was for circular pupils and for symmetric aberrations containing only the even terms of Zernike polynomial expansion (i.e., only cosine dependence), and presented some limitations such as its poor performance for large numerical apertures. Many of these shortcomings were fixed with the introduction of the vectorial ENZ theory for high numerical apertures \cite{Braat_etal2003}, and in more recent publications \cite{VanHaverPhD,VanHaver,VanHaver2014}.

In this paper we develop an alternative technique for calculating integrals of the form \eqref{fourier1} 
by combining the essence of the ENZ approach with the approximation power of the radial basis functions (RBF), which are standard tools in the approximation theory and numerical analysis, appearing in applications ranging from engineering and artificial intelligence to solutions of partial differential equations, see for example~\cite{Fasshauer}. The Gaussian radial basis functions (GRBF) have been also used in the context of ophthalmic optics \cite{andrei2009,andrei2011,Montoya}. 

The structure of the paper is as follows. In Section~\ref{sec:RBFapprox} we outline some facts from the approximation theory by RBF. The exposition here is very sketchy, and its main purpose is to mention some literature for the reader's convenience. Section~\ref{Fourier_from_GRBF} is devoted to the derivation of the main formulas that can be used for the calculation of general Fourier integrals \eqref{fourier1}. The detailed algorithm is developed in Section~\ref{sec:calculation} for the diffraction integrals defined therein, including its efficient implementation and error estimates.  An assessment of accuracy and efficiency of this procedure is carried out in Section~\ref{sec5}, where it is compared also with the extended Nijboer-Zernike theory (outlined in Section~\ref{sec:ENZ}) and with the standard 2D fast Fourier transform. 

Some of the ideas underlying this approach have been presented in~\cite{RMI2014}. In this work we develop further the algorithmic aspects of this procedure, prove its convergence, discuss its efficient implementation, and  compare its performance with some standard alternatives.

\section{Approximation by Radial Basis Functions} \label{sec:RBFapprox}

Every semi-analytic method for calculating integrals is comprised of two phases: first we must approximate the  integrand (or some of its factors) by a function from a given class, and after that the integral of the approximant is calculated in a closed form. This is the case of the ENZ approach (where the pupil function is approximated by linear combinations of Zernike polynomials, see the details in Section~\ref{sec:ENZ}), as well of the method put forward in this paper, which is based on approximation by a linear combination of Gaussian functions. 

A Gaussian radial basis function (or GRBF),
$$
\Phi(x,y)=e^{-\lambda(x^2+y^2)} 
$$
is a strictly positive definite radial basis function (RBF) in $\R^2$ with the \emph{shape parameter} $\lambda>0$ and center at the origin. As a consequence, any set of data $(a_k,b_k,z_k)$, with $P_k=(a_k, b_k)\in \R^2 $ pairwise distinct, is unisolvent for Lagrange interpolation by a linear combination of shifts $\Phi(x-a_k,y-b_k)$, that we briefly call Gaussian interpolant.  

The theory of approximation by RBF (both the theoretical and computational aspects) has been rather well developed in the last 30 years, motivated by its applications in numerical analysis (mesh-free methods) and neural networks, see e.g.~\cite{Buhman,Fasshauer}.

The accuracy of the Gaussian interpolants and the stability of the corresponding linear system depend on the number of data points and on the shape parameter $\lambda$. For a fixed $\lambda$, as the number of data points increases, the Gaussian interpolant converges to the underlying (sufficiently smooth) function being interpolated at a ``super-spectral'' rate $\mathcal O(e^{-c h^{-2}})$, where $h$ is the measure of a ``typical'' distance between the data points, and $c$ is a positive constant, depending on $\lambda$, see \cite{Schaback,Yoon}. The selection of the appropriate $\lambda$ is important: smaller values of $\lambda$ yield better accuracy of the interpolant but affect significantly the stability of the linear system (the ``trade-off principle''). Nonetheless, methods have been described to handle this situation  \cite{Fasshauer/McCourt,Larsson}.

Least-squares approximation by RBF has been developed in parallel within the artificial intelligence community. For instance, the support vector (SV) machine is a type of learning machine, based on statistical learning theory, which contains RBF networks as special cases. In the RBF case, the SV algorithm automatically determines centers, weights, and threshold that minimize an upper bound on the expected test error \cite{SV,Suykens}. This has connections with other successful approximation schemes, such as the adaptive least squares (see e.g.~\cite[Chapter 21]{Fasshauer}, and also \cite{andrei2011}), the moving least squares~\cite[Chapters 21--27]{Fasshauer}, and others.

The discrete least squares approximation to a smooth function is known to converge with at least a power rate; for some error bounds and implementation issues we refer the reader to \cite[Chapter 8]{Buhman} (see also \cite[Chapter 20]{Fasshauer}).

\section{Fourier integrals for Gaussian RBF} \label{Fourier_from_GRBF}


Consider   the case when $g$ in \eqref{fourier1} is a Gaussian RBF with the shape parameter $\lambda>0$ and center at a  point with cartesian coordinates $(a,b)$, multiplied by a radially-symmetric complex-valued function $h$; in other words, let
$$
g(\rho,\theta) =   h(\rho) e^{-\lambda ((x-a)^2+(y-b)^2)}  , \quad x = \rho \cos(\theta), \quad y = \rho \sin(\theta).
$$
Using the polar coordinates for the center,
$$
a  = q  \cos(\alpha ), \quad b  = q  \sin(\alpha ), 
$$
we can rewrite it as
\begin{equation}
\label{fespecial}
g(\rho,\theta) =   h(\rho)    e^{ - \lambda ( q^2 + \rho^2 - 2 \rho q \cos{(\theta - \alpha )}) }.
\end{equation}
In this section we discuss an expression for \eqref{fourier1} with $g$ as in \eqref{fespecial}:
\begin{equation} \label{firstF}
\begin{split}
F(r, \phi; q, \alpha)= & \frac{1}{\pi} \int_{0}^\infty\int_0^{2\pi} e^{ - \lambda ( q^2 + \rho^2 - 2 \rho q \cos{(\theta - \alpha )}) } e^{2\pi i r \rho \cos (\theta- \phi) } h(\rho) \rho d\rho d\theta.
\end{split}
\end{equation}
\begin{thm} \label{thm1}
For $F$ defined in \eqref{firstF},
\begin{equation} \label{FthroughLaplace}
F(r, \phi; q, \alpha)   =     e^{-\lambda q^2} \mathcal L\left[ I _0\left(2 \sqrt{\cdot \Omega} \right) h(\sqrt{\cdot})  \right](\lambda),
\end{equation}
where 
\begin{equation} \label{omega}	
\Omega = \Omega(r,\phi; q,\alpha) =  \lambda^2 q^2  + 2\pi i r \lambda q \cos (\phi-\alpha)-\pi^2 r^2 ,
\end{equation}
$I_0$ is the modified Bessel function of the first kind, and $\mathcal L[\cdot]$ denotes the Laplace transform. 
\end{thm}
\begin{proof}
	Notice that by \eqref{firstF},
\begin{equation} \label{firstFalt}
F(r, \phi; q, \alpha)  
=  \frac{1}{\pi}e^{-\lambda q^2} \int_0^\infty d\rho  \, h(\rho) \rho e^{-\lambda \rho^2}    \int_0^{2\pi} e^{    2 \lambda   \rho q \cos{(\theta - \alpha )} } e^{2\pi i r \rho \cos (\theta- \phi) } d\theta .
\end{equation}
For arbitrary constants $A, B\in \C$ we have the following identity:
\begin{equation} \label{I}	
  \int_0^{2\pi} e^{2A \cos (\theta-\alpha)} e^{2B \cos(\theta)}d\theta 
=2\pi I _0\left(2\sqrt{A^2 + 2 A B \cos \alpha+B^2}\right).
\end{equation}
Observe that $I_0$ is an entire even function, so that the value in the right hand side of \eqref{I} is independent of the choice of the branch of the square root.

Identity \eqref{I} is  a straightforward consequence of the fact that
$$
A \cos (\theta-\alpha) + B \cos(\theta) = \sqrt{A^2 + 2 A B \cos \alpha+B^2} \cos (\theta - \gamma), 
$$
where
$$
\gamma = \arctan \left( \frac{A\sin \alpha}{B + A \sin \alpha}\right),   
$$
(with an appropriate choice of the branch of the square root), along with the identity (\cite{abramowitz}, formula (9.6.16)):
$$
\mathop{I_{0}\/}\nolimits\!\left(z\right)=\frac{1}{\pi}\int_{0}^{\pi}e^{\pm z%
\mathop{\cos\/}\nolimits\theta}d\theta.
$$

Using \eqref{I} in \eqref{firstF} we conclude that
\begin{equation} \label{secondF}
\begin{split}
F(r, \phi; q, \alpha) & =    2 e^{-\lambda q^2} \int_0^\infty   e^{-\lambda \rho^2}    I _0\left(2\rho \sqrt{\Omega} \right) h(\rho)   \rho\,  d\rho   \\
& =    e^{-\lambda q^2} \int_0^\infty   e^{-\lambda \rho}    I _0\left(2 \sqrt{\rho \Omega} \right) h(\sqrt{\rho})   \,  d\rho ,
\end{split}
\end{equation}
which is equivalent to \eqref{FthroughLaplace}.
\end{proof}

Formulas in Theorem~\ref{thm1} are the basic building blocks for the algorithm proposed below.
Let us particularize these identities to the  case of a circular exit pupil of an optical system, when $h$ can be taken as the characteristic function of the interval $[0,1]$:
\begin{cor} \label{cor2}
Let
$$
h(\rho)=\chi_{[0,1]}(\rho)=\begin{cases}
1, & \text{if } 0\leq \rho\leq 1, \\
0, & \text{otherwise.}
\end{cases}
$$
Then $F$ in \eqref{firstF} takes the form
\begin{equation} \label{seriesF1}
F(r, \phi; q, \alpha) =   e^{-\lambda  q ^2}  \sum_{s=0}^\infty \frac{m_s(\lambda )}{(s!)^2} \,  \Omega(r,\phi;q ,\alpha )^{s},
\end{equation}
where 
\begin{equation} \label{defM1}
m_s(\lambda)=\begin{cases}
\displaystyle  \int_0^1   e^{-\lambda \rho} \rho^s d\rho, & \lambda \neq 0,\\
(1+s)^{-1}, & \lambda = 0.
\end{cases}
\end{equation}
Furthermore, $|m_s(\lambda)|\leq 1$ for $\Re(\lambda)\geq 0$, and the series in \eqref{seriesF1} is uniformly convergent for all values of $r, q\geq 0$, $\phi, \alpha\in \R$ and $\lambda>0$.
\end{cor}

\begin{rmk}
Integration by parts and straightforward calculations show that for $s=0, 1, 2, \dots$ and $\lambda\neq 0$,
\begin{equation} \label{recurr1}
m_0(\lambda)=\frac{1-e^{-\lambda} }{\lambda}, \quad m_{s+1}(\lambda)=\frac{(s+1) m_s(\lambda)- e^{-\lambda}}{\lambda}.
\end{equation}
Notice also that
\begin{equation}
\label{m_kderiv}
m_0(\lambda)=\mathcal L[\chi_{[0,1]}](\lambda), \quad m_s(\lambda)=\left( - \frac{d}{d\lambda}\right)^{s-1} m_0(\lambda), \quad s\geq 1.
\end{equation}
\end{rmk}
\begin{proof}
By \eqref{FthroughLaplace}, the crucial step is the evaluation of the integral
$$
\int_0^1   e^{-\lambda \rho }    I _0\left(2 \sqrt{\rho \Omega} \right)    d\rho  .
$$
The Taylor series for the Bessel function (see \cite{abramowitz}, formula (9.6.12)),
\begin{equation*}
I _0\left(2 \sqrt{\rho\Omega} \right) =\sum_{s=0}^\infty \frac{\Omega^{s}}{(s!)^2} \rho^s,
\end{equation*}
is locally uniformly convergent on the whole plane, and thus
\begin{equation}
\label{integralMain1}
\int_0^1   e^{-\lambda \rho }    I _0\left(2 \sqrt{\rho \Omega} \right)    d\rho = \sum_{s=0}^\infty \frac{\left( m_s(\lambda) \Omega\right)^{s}}{(s!)^2} .
\end{equation}
Furthermore, the bound for $m_s(\lambda)$ is trivial for purely imaginary $\lambda$, while for  $\Re(\lambda)> 0$, 
\begin{align*}
|m_s(\lambda)|& \leq |m_s(\Re(\lambda))| \leq  \int_0^1 e^{-\Re(\lambda) \rho}   d \rho= |m_0(\Re(\lambda))|=\frac{  1 - e^{-\Re (\lambda)}  }{\Re (\lambda)} \leq 1,
\end{align*}
which concludes the proof.
\end{proof}

Clearly, formulas of Theorem~\ref{thm1} and Corollary~\ref{cor2} can be easily extended to the case when $g$ is a linear combination of functions of the form \eqref{fespecial},  that is,
\begin{equation*}
g(\rho,\theta) =  h(\rho)   \sum_{k=1}^K c_k  e^{ - \lambda_k ( q_k^2 + \rho^2 - 2 \rho q_k \cos{(\theta - \alpha_k )}) }, \quad h(\rho)= \chi_{[0,1]}(\rho) ,
\end{equation*}
where $c_k\in \C$ are certain coefficients, $\lambda_k >0 $ are the shape parameters, and  $(q_k, \alpha_k)$ are the polar coordinates of the corresponding centers of the Gaussian RBFs.  The corresponding Fourier transform $F$, defined in \eqref{fourier1}, can be written as 
$$
F(r, \phi)= \sum_{k=1}^K c_k e^{-\lambda_k q_k^2}  \sum_{s=0}^\infty \frac{m_s(\lambda_k)}{(s!)^2} \,  \Omega(r,\phi;q_k,\alpha_k)^{s},
$$
with $\Omega$ defined in \eqref{omega}.

If we assume additionally that all the shape parameters are equal,
$
\lambda_1=\dots=\lambda_K=\lambda,
$
we can rearrange the calculations as follows:
\begin{equation*}
F(r, \phi)= \sum_{s=0}^\infty  \frac{m_s(\lambda)}{(s!)^2} \, H_s(r,\phi), \quad H_s(r,\phi)=\sum_{k=1}^K c_k \,e^{-\lambda q_k^2}  \, \Omega(r,\phi;q_k,\alpha_k)^{s}.
\end{equation*}

\section{Calculation of diffraction integrals} \label{sec:calculation}

Optical systems can be described by means of the complex-valued pupil function $P(\rho,\theta)$ that, in the case of a circular pupil, can be  expressed in (normalized) polar coordinates as
\begin{equation}\label{pupil}
P(\rho,\theta) = A(\rho,\theta) \exp\left( - i \frac{2 \pi n}{ \ell } W (\rho,\theta) \right) ,
\end{equation}
where $A(\rho,\theta)$ is the aperture or amplitude transmittance function of the optical system,  $W(\rho,\theta)$ is  the wavefront error, and constants $n$  and $\ell$ denoting  the refractive index and the wavelength of the light, respectively. 
According to Fourier optics~\cite{BornWolf}, the complex-valued point-spread function of such a system is given by the diffraction integral
\begin{equation}\label{dif_intBis}
U(r, \phi ;f) = \frac{1}{\pi} \int_0^1 \int_0^{2\pi} F(\rho,f) P(\rho, \theta) \exp\left(2\pi i \rho r \cos(\theta-\phi)\right) \rho d\theta d\rho ,
\end{equation}
where $f$ represents the defocussing ($f=\pi/2$ corresponds to one focal depth), and $(r,\phi)$ denote the polar coordinates in the image plane. The focal factor $F$ is
\begin{equation}\label{defocussing}
F(\rho,f)=\exp\left(i f \frac{1-\sqrt{1-s_0^2 \rho^2}}{1-\sqrt{1-s_0^2}}\right),
\end{equation}
where $0<s_0<1$ is the numerical aperture (NA) of the optical system. For small values of $s_0$ (say, $s_0<0.5$) the focal factor is well approximated by $\exp(if \rho^2)$, and the resulting integral is
\begin{equation}\label{dif_int}
U(r, \phi ;f) = \frac{1}{\pi} \int_0^1 \int_0^{2\pi} \exp(i f \rho^2) P(\rho, \theta) \exp\left(2\pi i \rho r \cos(\theta-\phi)\right) \rho d\theta d\rho ,
\end{equation}
known as the Debye approximation of the Rayleigh integral of the optical impulse response or the point-spread function (PSF) of the system, evaluated by
\begin{equation*}
PSF(r, \phi ;f) = | U(r, \phi ;f) |^2.
\end{equation*}

In this section we discuss the computational framework and the efficient implementation of the algorithm for calculation of such diffraction integrals for small NA, using \eqref{dif_int}, although our method easily carries over to any NA and to the focal factor \eqref{defocussing}. 

Observe that formally the defocussing term ($\exp(i f \rho^2)$) in   \eqref{dif_int} could be absorbed in the aberration term (pupil function). However, from the physical point of view, the aberration term is an intrinsic error of the optical system, while the defocussing is a deliberately introduced defect, which can take on several values that are relatively large with respect to the aberrations of the system. For this reason, a separate treatment is commonly preferred, see \cite{Braat_etal2002}.

\subsection{Computational framework }\label{computFrame}

In practice,  the wavefront is sampled at a discrete and finite set of points on the pupil (the unit disk). In other words, the input data set has the form $(x_j,y_j,W_j)$, with $W_j = W(x_j,y_j)$. Values $W_j$ can be measured directly  or obtained by standard procedures. 

As it was mentioned above, semi-analytic methods for calculation of Fourier  \eqref{fourier1} or diffraction integrals \eqref{dif_int} are based on an approximation of the integrand by a function of a suitable chosen form. The brief outline of the results in Section~\ref{sec:RBFapprox} and the reasonable assumption that the complex pupil function $P$ in \eqref{pupil} is sufficiently smooth allows us replace the actual pupil function by its approximation by a linear combination of Gaussian radial basis functions (GRBF). It is convenient to point out that the implementation of this approximation and the achieved accuracy are not subject of this paper, and are treated in the literature, see Section~\ref{sec:RBFapprox}.

Hence, our starting point for calculation of \eqref{dif_int} is the ansatz that the complex pupil function $P$ is given by an expression of the form
\begin{equation}
\label{basisExp}
P(\rho,\theta) =   \sum_{k=1}^K c_k  g_k(\rho,\theta) ,  \quad g_k(\rho,\theta) = e^{ - \lambda_k ( q_k^2 + \rho^2 - 2 \rho q_k \cos{(\theta - \alpha_k )}) },  
\end{equation}
where $c_k\in \C$ are certain coefficients, $\lambda_k >0 $ are the shape parameters, and  $(q_k, \alpha_k)$ are the polar coordinates of the  centers of the Gaussian RBFs.  
A practical way to obtain \eqref{basisExp} is to fix a basis   $\{ g_k \}_{k=1}^K$ of GRBF, or equivalently, to choose, for each index $k$, values for $q_k$, $\alpha_k$ and $\lambda_k$. 
The optimal choice of the shape parameter based on the given data is a highly non-linear problem, usually solved by cross validation, see \cite{Fasshauer/Zhang,Rippa}. Using this method we arrived at the following rule of thumb  been tested in practice: take the same shape parameter for all $k$'s,
\begin{equation}
\label{lambdaequal}
 \lambda_1=\dots=\lambda_K=\lambda\in [10,20],
\end{equation}
select $r\in \N$ and  create a grid of $r \times r$ equally spaced  points in the square $[-1.2, 1.2]^2$. The goal of covering an area larger than the unit disk is to deal with the Gibbs phenomenon at the boundary of the disk. The polar coordinates of these  $K=r^2$ points will be the pairs $(q_k, \alpha_k)$ in \eqref{basisExp}. 

With the basis $\{g_k\}$ chosen, we can calculate the coefficients $c_k$ fitting the pupil function by means of the linear least squares. For a high number of centers $K$ this problem is ill-conditioned, and the use of  Tikhonov regularization \cite{regularization} is recommended, with an appropriate choice of regularization parameter, for instance, by the $L$-curve method or by Morozov discrepancy principle. In our experiments, implemented in Matlab, we used the algorithms described in~\cite{calvetti}, see also \cite{Lijun Li}.  For the error estimates, see the references in Section~\ref{sec:RBFapprox}. 

An alternative approach to the 2D nonlinear approximation of functions by sums of exponentials is described in  \cite{Andersson2010}.

With all the parameters in the representation \eqref{basisExp} calculated  we can use the formulas obtained in Section \ref{Fourier_from_GRBF}, namely 
\begin{equation} \label{U1}	
U(r, \phi ;f)= \sum_{k=1}^K c_k \,e^{-\lambda_k q_k^2} \sum_{s=0}^\infty \frac{m_s(\lambda_k-if)}{(s!)^2} \,  \Omega(r,\phi,q_k,\alpha_k)^{s},
\end{equation}
with
\begin{equation*} 
\Omega(r,\phi; q,\alpha) =  \lambda^2 q^2  + 2\pi i r \lambda q \cos (\phi-\alpha)-\pi^2 r^2 .
\end{equation*}
Furthermore, with the additional assumption \eqref{lambdaequal} we can rewrite this formula as
\begin{equation}
\label{defH}
U(r, \phi ;f)= \sum_{s=0}^\infty  \frac{m_s(\lambda-if)}{(s!)^2} \,  H_s(r,\phi), \quad H_s(r,\phi)=\sum_{k=1}^K c_k \,e^{-\lambda q_k^2}  \, \Omega(r,\phi,q_k,\alpha_k)^{s}.
\end{equation}

\subsection{Efficient implementation of the calculations} \label{evaluation_grbf}

The recurrence relation \eqref{recurr1} can present some numerical instability, especially for large value of the imaginary part of the argument; these considerations suggest to replace the coefficients $m_s(\lambda)$, defined in \eqref{defM1}, by their normalized counterparts, 
\begin{equation*}
\widehat m_s(\lambda) =\frac{m_s(\lambda)}{(s!)^2},
\end{equation*}
that can be efficiently generated for $\lambda\neq 0$ by the following double recurrence:
\begin{equation}
\label{rec_hatmk}
  \widehat m_{s+1}(\lambda)=\frac{1}{\lambda} \left[ \frac{\widehat m_s(\lambda)}{s+1} - \tau_s(\lambda)\right]  , \quad  \tau_{s+1}(\lambda) = \frac{\tau_s(\lambda)}{(s+2)^2}, \quad s=0, 1, 2, \dots
\end{equation}
with the initial conditions
\begin{equation}
\label{rec_hatmkInit}
\tau_0(\lambda)= e^{-\lambda},\quad \widehat m_0(\lambda)=\frac{1-\tau_0(\lambda)}{\lambda}.
\end{equation}

Since formulas in \eqref{U1}--\eqref{defH} contain an infinite sum, we must choose  a \emph{cut-off parameter}, $S\in \N$, so that the expression in \eqref{U1} is replaced by
\begin{equation} \label{U1truncated}
U(r, \phi ;f)= \sum_{k=1}^K c_k \,e^{-\lambda_k q_k^2} \sum_{s=0}^S \widehat  m_s(\lambda_k-if) \,  \Omega(r,\phi,q_k,\alpha_k)^{s},
\end{equation}
while \eqref{defH} becomes
\begin{equation} \label{U1truncated_withH}
U(r, \phi ;f)= \sum_{s=0}^S  \widehat  m_s(\lambda_k-if) \,  H_s(r,\phi),
\end{equation}
where we denote 
\begin{equation*} 
H_s(r,\phi)=\sum_{k=1}^K c_k \,e^{-\lambda_k q_k^2}  \, \Omega(r,\phi,q_k,\alpha_k)^{s}.
\end{equation*}
It is worth observing that in these formulas the defocus parameter $f$ and the coordinates  $r$ and $\phi$ are independent: $f$ does not appear in $\Omega$, while  $m_s (\lambda_k- i f )$ need to be calculated only once regardless the number of points at which we evaluate the functions.

In practice, we want to find the values of $U(r, \phi ;f)$ at a vector of $J$ points on the plane, given by their polar coordinates $(r, \phi)$, and for a vector of defocus parameters $f$. In other words, 
we need to compute efficiently the matrix
$$
\bm U = \left(  U(r_j, \phi_j ;f_m) \right)_{m,j} \in \C^{M\times J}.
$$
Let us discuss the algorithm under the assumption \eqref{lambdaequal}, so we will use formula~\eqref{defH}.

Applying \eqref{rec_hatmk}--\eqref{rec_hatmkInit} to the vector $\bm f$, we obtain the matrix 
\begin{equation*}
\bm M = \left(\widehat m_s(\lambda- i f_m) \right)_{m=1, \dots, M}^{s=0, 1, \dots, S} \in \C^{M \times (S+1)},
\end{equation*}
which, as it was mentioned, is computed once for all $(r_j, \phi_j)$'s. 

Another observation that will speed up computations is that if $P$ and $Q$ are $\R^{2\times 1}$ vectors of Cartesian coordinates of points on $\R^2$ with polar coordinates $ (r,\phi)$ and $ (q,\alpha)$, respectively, then by   \eqref{omega},
\begin{equation} \label{omegaBis}	
\begin{split}
  \Omega(r,\phi; q,\alpha) & =  \lambda^2 q^2  + 2\pi i r \lambda q \cos (\phi-\alpha)-\pi^2 r^2   = ( \lambda Q + \pi i P)^T   ( \lambda Q + \pi i P) ,
  \end{split}
\end{equation}
where the superscript $T$ denotes the transpose of a matrix. In other words, $  \Omega(r,\phi; q,\alpha)$ is the square of the euclidean distance (in $\R^2$) from $\lambda Q$ to $-\pi i P$. It allows us to encode efficiently the evaluation of  $\bm \Omega=(\Omega_{k,j}) \in \R^{K\times J}$,
$$
\Omega_{k,j}= \Omega(r_j,\phi_j,q_k,\alpha_k),
$$
by first finding the Cartesian coordinates of the centers,
$$
a_k= q_k \cos\alpha_k, \quad b_k=q_k \sin \alpha_k, \quad k=1, \dots, K.
$$ 
Finally, we find  the column vector 
\begin{equation}
\label{rowD}
\bm d = (c_k \,e^{-\lambda q_k^2})_{k=1, \dots, K} \in \R^{K\times 1}.
\end{equation}
We describe the rest of the procedure in the Algorithm~\ref{algo1}.

\IncMargin{1em}
\begin{algorithm}
  \SetKwInOut{Input}{input}\SetKwInOut{Output}{output}
  \Input{Matrices $\bm M   \in \C^{M \times (S+1)}$, $\bm \Omega=(\Omega_{k,j}) \in \R^{K\times J}$, and $\bm d  \in \R^{K\times 1} $ }
  \Output{Matrix $\bm U  \in \C^{M\times J}$}
  \BlankLine
  \emph{Initialization:}
 
 Set $\bm U =\bm 0\in \R^{M \times J}$
 
 Define $\bm R\in \R^{K\times J}$ as the matrix of $1$'s, and compute 
 \begin{equation}
\label{calculoH}
\bm H = \bm d^T \bm  R.
\end{equation}

 \tcp*[f]{At this stage,  $\bm H=H_0(r_j,\phi_j)_{j=1, \dots, J^2}\in \C^{1\times J}$}
 
  \BlankLine
  \For{$s\leftarrow 0$ \KwTo $S-1$}{
      Calculate 
  \begin{equation*}
\bm U \leftarrow \bm U + \bm M(:, s+1)  \bm H,
\end{equation*}
  Update $\bm R$ by
$$
\bm R \leftarrow \bm R .* \bm \Omega
$$
\tcp*[f]{We use  the Matlab notation  (\texttt{.*}) for the term-by-term multiplication}\;
Compute $\bm H$ using \eqref{calculoH}\;
  }
  \caption{Implementation of formula~\eqref{defH}.}\label{algo1}
\end{algorithm}\DecMargin{1em}

It should be noticed finally that we can slightly increase the accuracy adding a constant term to the approximation of the complex pupil function, that is, instead of \eqref{basisExp} using the representation 
\begin{equation*}
P(\rho,\theta) =  c_0+ \sum_{k=1}^K c_k  g_k(\rho,\theta) ,  \quad g_k(\rho,\theta) = e^{ - \lambda_k ( q_k^2 + \rho^2 - 2 \rho q_k \cos{(\theta - \alpha_k )}) },  
\end{equation*}
where $c_0$ can be taken as the average of the values of $P$ at the given points. Notice that the contribution to $U$ corresponding to this term can be evaluated using Algorithm~\ref{algo1} with $\lambda=0$.

\subsection{Error estimates and convergence analysis} \label{sec:error estimates}

The algorithm described above presents two main sources of errors (besides the standard truncation and machine arithmetic errors that we do not discuss here):
\begin{itemize}
\item the fitting error in \eqref{basisExp};
\item the truncation error, given by the choice of $S$, in replacing the infinite series in \eqref{U1} or \eqref{defH} by finite sums.
\end{itemize}
The former is not subject of this paper, and the reader is referred to the discussion in Section~\ref{sec:RBFapprox}.

Regarding the the truncation error, observe that by Corollary~\ref{cor2},  for $\lambda>0$, 
$$
|m_s(\lambda-if)|\leq 1.
$$
On the other hand, by expression \eqref{omega},
$$ 
|\Omega|=|\Omega(r,\phi; q,\alpha)|  =\sqrt{ \left( \lambda^2 q^2  -\pi^2 r^2\right)^2+ \left( 2\pi   r \lambda q \cos (\phi-\alpha)\right)^2
}=    \lambda^2 q^2  +\pi^2 r^2 .
 $$
In consequence, the truncation error in \eqref{U1} is given by
\begin{align*}
\left| \sum_{s=S+1}^\infty \frac{m_s(\lambda_k-if)}{(s!)^2} \,  \Omega(r,\phi,q_k,\alpha_k)^{s}\right| & \leq \sum_{s=S+1}^\infty \frac{1}{(s!)^2} \,  \left(\lambda_k^2 q_k^2  +\pi^2 r^2  \right)^{s},
\end{align*}
or in other words, the remainder of the Taylor expansion for the modified Bessel function
$$
I_0\left(2 \sqrt{\lambda_k^2 q_k^2  +\pi^2 r^2} \right),
$$
and thus,
$$
\left| \sum_{s=S+1}^\infty \frac{m_s(\lambda_k-if)}{(s!)^2} \,  \Omega(r,\phi,q_k,\alpha_k)^{s}\right|  \leq \frac{\left(\lambda_k^2 q_k^2  +\pi^2 r^2  \right)^{S+1}}{(S+1)!} \max_{0\leq t\leq \lambda_k^2 q_k^2  +\pi^2 r^2} \left| I^{(S+1)}_0\left(t \right)\right|.
$$

For the practical implementation, in our algorithm with \eqref{lambdaequal}, $q\leq 1.7$, $r\leq 2$, $\lambda\leq 20$, so that
\begin{equation}
\label{upperbound}
|\Omega|\leq 1200.
\end{equation}
This bound is very conservative, but shows that in the worst case scenario $S=100$ provides a truncation error of order $10^{-9}$.

Although, as the numerical experiments explained in Section~\ref{sec5} show, the method performs well even for reasonably large values of $r$ and $f$ in \eqref{dif_int}, we can slightly improve accuracy by using the scheme \eqref{U1}, and replacing the infinite series by its Pad\'e approximant of a suitable order with center at the origin (instead of a mere truncation, as we did above), see e.g.~\cite{gravesmorris}.

\section{Outline of the extended Nijboer-Zernike theory} \label{sec:ENZ}

In Section~\ref{sec5} we are going to produce the results of some numerical experiments regarding the behavior of the algorithm described above, and to compare them with the standard procedures used for the computation of diffraction integrals.

One of the best-known semi-analytic methods of calculation of these integrals is the so-called Extended Nijboer--Zernike (ENZ) theory~\cite{Braat_etal2002,Janssen2002}, similar in spirit to the method presented in this paper, with the main difference that  the pupil function $P$ is expanded in terms of Zernike polynomials, instead of Gaussian radial basis functions. 

Recall that the $(n,m)$ Zernike polynomial is defined as
$$ Z_n^m (\rho,\theta) =  \left\{ 
\begin{array}{ll} 
N_n^m R_n^{|m|}(\rho) \cos(m \theta), & m\geq 0 ,
\\  N_n^m R_n^{|m|}(\rho) \sin(|m| \theta), &m < 0 ,
\end{array} \right. $$
where $0 \leq \rho \leq 1$ and $0 \leq \theta  \leq 2 \pi$ are polar coordinates, and the double index $(n,m)$ is such that $n\geq 0$, $m \leq |n|$, and $n-m$ is an even number. The radial part $R_n^{|m|}(\rho)$ is a rescaled Jacobi polynomial $P_s^{\alpha,\beta}$,
$$ R_n^{|m|}(\rho) = (-1)^{(n-m)/2} \rho^m P_{(n-m)/2}^{(m,0)}(1-\rho^2 ),
$$ 
given explicitly by 
\begin{equation} \label{Zernike_radial} R_n^{|m|}(\rho) = \sum_{s=0}^{ (n-|m|)/2 } \frac{ (-1)^s (n-s)!}{ s! ((n+|m|)/2 - s)! ((n-|m|)/2 - s)!} \rho^{n-2s}.
\end{equation}
With the normalization constant $N_n^m = \sqrt{ (2-\delta_{0,m})(n+1)}$ the Zernike polynomials become orthonormal on the unit disk $\mathbb D = \{ (x,y) \in \mathbb R^2 / x^2+y^2 \leq 1\}$ with respect to the area measure (see e.g.~\cite[Section 9.2]{BornWolf} for further properties of Zernike polynomials).

For the reader's convenience, in this section we present an outline of the ENZ formulas. 
Assume that (with a suitable approximation)
\begin{equation} \label{expression:pupil:ENZ}
P(\rho,\theta) = \sum_{n,m} c_n^m Z_n^m(\rho,\theta).
\end{equation}
If the coefficients $c_n^m$ in the expression \eqref{expression:pupil:ENZ} are known, then the integral $U$ in \eqref{dif_int} 
takes the form
\begin{equation} \label{expression:ENZ:Utotal}
 U(r,\phi;f) = \sum_{n,m} c_n^m U_n^m(r,\phi;f),
\end{equation}
where
$$ U_n^m (r,\phi;f) = \frac{1}{\pi} \int_0^1 \int_0^{2\pi}\left[\exp(i f \rho^2) Z_n^m(\rho,\theta) \exp\left((2\pi i \rho r \cos(\theta-\phi)\right)\right]\rho d\theta d\rho . $$
According to the original ENZ theory, for the double-index $(n,m)$ with $m \geq 0$ (a necessary condition for the applicability of these formulas) we have
\begin{equation}\label{Unm}
 U_n^m(r,\phi; f) = 2 i^m V_n^m(r,f) \cos( m \phi ),
\end{equation}
with  
\begin{equation}\label{Vnm}
V_n^m (r, f) = \exp( i f ) \sum_{k=0}^{\infty} \left( \frac{ - i f }{ \pi r} \right)^k B_k(r)
\end{equation}
and  
\begin{equation}\label{B_k}
B_k(r) =  \sum_{j=0}^p u_{k j} \frac{ J_{m+k+2j+1} (2 \pi r) }{2 \pi r},
\end{equation}
with the coefficients
\begin{equation}\label{enz_coefs}
u_{k j} = (-1)^p \frac{m+k+2j+1}{q+k+j+1} \binom{ m+k+j}{k} \binom{k+j}{k} \binom{k}{p-j} \left/ \binom{q+k+j}{k} \right.
\end{equation}
where $p = (n-m)/2$ and $q=(n+m)/2$.
Here $J_\nu$ is the  Bessel  function of the first kind and order $\nu$, see e.g.~\cite[Chapter 9]{abramowitz}, and formulas \eqref{Vnm}--\eqref{B_k} are known as the \emph{power-Bessel series} (ENZ-PB) expression for $V_n^m$.

Observe that these formulas contain a removable singularity at $r=0$, so some extra care when organizing the calculations should be put. Moreover, they also present difficulties when the defocus parameter $f$ is large. According to \cite{Braat_etal2002}, approximately $3|f|$ terms are needed in general to accurately evaluate expression \eqref{Vnm}, and the practical use of that formula is limited to a range $|f| \leq 5 \pi$. 

In order to overcome these difficulties an alternative approach has been put forward in \cite{Janssen2004}. The main idea is the use of  Bauer's identity:
\begin{equation} \label{equation:Bauers}
\exp(i f \rho^2) = \exp\left(\frac{1}{2} i f \right) \sum_{k=0}^{\infty} (2k+1) i^k j_k \left( \frac{1}{2} f \right) R_{2k}^0 (\rho), 
\end{equation}
where $R_n^m$ is the radial part of the Zernike polynomials,  and 
\begin{equation} \label{equation:Bessel:j}
 j_k(z) = \left( \frac{\pi}{2z} \right)^{1/2} J_{k+(1/2)} (z), \quad k=0,1,\dots
\end{equation}

Using \eqref{expression:pupil:ENZ} and \eqref{equation:Bauers} in  \eqref{dif_int} we reduce the problem to the computation of integrals containing products of two radial Zernike polynomials, $R_{n_1}^{m_1}$ and $R_{n_2}^{m_2}$. The crucial idea is to use the linearization formulas for these products of the form 
$$ R_{m_1+2p_1}^{m_1} R_{m_2+2p_2}^{m_2} = \sum_{l} c_l R_{m_1+m_2+2l}^{m_1+m_2}, $$
where coefficients $c_l$ defined as
$$ c_l = \sum_{s_1,s_2,t} f_{p_1,s_1}^{m_1} f_{p_2,s_2}^{m_2} g_{s_1+s_2-2t,l}^{m_1+m_2}  $$
 are expressed in terms of the quantities $f_{p,s}^m$ and $g_{u,l}^m$ given explicitly by
\begin{equation} \label{expression:ENZ:coeff1}
f_{p,s}^m = (-1)^{p-s} \frac{2s+1}{p+s+1} \left[  \binom{m+p-s-1}{m-1} \binom{m+p+s}{s} \left/ \binom{p+s}{s} \right. \right], 
\end{equation}
for $s=0,\hdots,p$, and
\begin{equation} \label{expression:ENZ:coeff2}
g_{u,l}^m = \frac{m+2l+1}{m+u+l+1} \left[ \binom{m}{u-l} \binom{u+l}{l} \left/ \binom{m+l+u}{m+l} \right. \right],
\end{equation}
for $u=l,\hdots,l+m$. For $m=0$, expressions \eqref{expression:ENZ:coeff1}  and \eqref{expression:ENZ:coeff2} boil down to 
$$ f_{p,s}^0 = \delta_{p,s}, \quad g_{u,l}^0 = \delta{u,l} , $$
where $\delta$ is the Kronecker's delta.

As a result, we obtain  the so-called \textit{Bessel-Bessel series} (ENZ-BB) expression for $V_n^m$ in \eqref{Vnm}: for $n$, $m$  nonnegative integers with $n-m \geq 0$ and even, with $p=\frac{1}{2}(n-m)$ and $q=\frac{1}{2}(n+m)$,
\begin{equation} \label{Vnm_improved}
V_n^m(r,f)=\exp\left(\frac{1}{2} i f\right) \sum_{k=0}^{\infty} (2k+1) i^k j_k \left(\frac{1}{2}f \right) \sum_{l= l_0  }^{k+p} (-1)^l w_{k,l} \frac{ J_{m+2l+1}(2 \pi r)}{2 \pi r},
\end{equation}
where $l_0 = \max\{0, k-q, p-k\}$ and  
\begin{equation} \label{expression:ENZ:coeff4}
w_{k,l} = \sum_{s=0}^p \sum_{t=0}^{\min\{k,s\}} f_{p,s}^m b_{k,s,t} g_{k+s-2t,l}^m, 
\end{equation}
with
\begin{equation} \label{expression:ENZ:coeff3}
b_{s_1,s_2,t} = \frac{2s_1 + 2s_2 - 4 t +1 }{2s_1+2s_2-2t+1} \left( \frac{ A_{s_1-t} A_t A_{s_2-t} }{ A_{s_1 + s_2 - t} } \right),
\end{equation}
for $t=0, \hdots, \min\{s_1,s_2\}$ and $A_k = \binom{2k}{k}$.

In the special case of $m=0$ we have that 
$$ 
w_{k, k+p-2j} = b_{k,p,j}, \quad j=0,1,\hdots, \min\{k,p\}, 
$$
while all other $w_{k,l}$ vanish. In this way, all $w_{k,l} \geq 0$ and $\sum_{l} w_{k,l} = 1$. 

The conclusion of \cite{Janssen2004} is that this new scheme is valid and accurate even for very large values of $f$. By analyticity, the expression \eqref{Vnm_improved}  remains valid also for complex values of $f$, which is important in a number of practical applications, with the downside in the removable singularity of \eqref{equation:Bessel:j}  at  $f=0$. In practice, it is convenient to use both ENZ schemes if the defocus ranges from $0$ to large values of $f$.


Further enhancement of the Bessel-Bessel approach was suggested recently in \cite{VanHaverPhD,VanHaver,VanHaver2014}, where equation~\eqref{Vnm_improved} is rearranged for its evaluation in the following way:
\begin{equation} \label{ENZ_EBB_Vnm}
V_n^m(r,f)=  \sum_{k=0}^{\infty}  C_{2k}^0(f) \sum_{\substack{h=-n+2k, \\ h -m \text{ even}, \\ h\geq|m|}}^{n+2k}
(-1)^{\frac{m-h}{2}} A_{2k,n,h}^{0,m,m} \, B_h (r).
\end{equation}  
Functions $C$ and $B$ of the defocus parameter $f$ and of the radius $r$, respectively, are given by:
\begin{align} \label{ENZ_EBB_C}
C_{2k}^0(f) &= (2k+1) i^k  \exp\left(\frac{1}{2} i f\right) j_k \left(\frac{1}{2}f \right),\\ 
\label{ENZ_EBB_B}
B_h (r) &=  \frac{ J_{h+1}(2 \pi r)}{2 \pi r},
\end{align}  
and the coefficients $A_{n_1,n_2,n_3}^{m_1,m_2,m_3}$ are defined as
\begin{equation} \label{ENZ_EBB_A}
A_{n_1,n_2,n_3}^{m_1,m_2,m_3} =  (-1)^{(n_1-n_2-m_3)/2} \sqrt{n_3+1} \left( \begin{array}{ccc} \frac{1}{2} n_1 & \frac{1}{2}  n_2 & \frac{1}{2} n_3 \\  \frac{1}{2}  m_1 &  \frac{1}{2}  m_2 & -  \frac{1}{2}  m_3 \end{array} \right),
\end{equation}  
where the term in parenthesis above represents the Wigner 3j-symbol (see \cite[Chapter 34]{NIST}), and is a real number. These expressions constitute the \emph{``enhanced'' Bessel--Bessel series} (ENZ-EBB) expression for $V_n^m$ in \eqref{Vnm}.

\section{Comparative assessment of accuracy and efficiency} \label{sec5}

In this section, the performance of the methods based on the Gaussian radial basis functions (GRBF) proposed in this paper is analyzed and compared to the popular alternative approaches. 

Since the closed analytic expression for a Fourier transform type integrals like \eqref{fourier1} and \eqref{dif_int} is possible only for most elementary integrands, for their computation we must rely either on numerical or on semi-analytical methods (or analytical approximations). The first group comprises  quadrature-type numerical procedures in which the integration over a 2D domain is replaced by the calculation of a discrete sum based on evaluation of the integrand at a discrete set of points, with the particular challenge of integrating a highly oscillatory function. The most popular approach is the bi-dimensional discrete Fourier transform (we will refer to it as the FFT2-based approach), calculated via the Fast Fourier Transform algorithm. Its efficiency can be substantially enhanced using the fractional Fourier transform \cite{BaileySwartztrauber1994} or a ``butterfly diagram'' ideas \cite{Candes2009}, see also \cite{Gai2007}.

The best known representative of  the second group is the Extended Nijboer--Zernike (ENZ) theory that was briefly exposed in Section~\ref{sec:ENZ}.  

In this section, these alternatives are discussed and compared with our method, paying special attention to their computational complexity, precision, accuracy and  speed. We use the ``naive'' notion of complexity, understanding by this the number of real floating point operations (\emph{flops}) needed to run the algorithm. Since the exact number of flops is in general difficult or not feasible to calculate, the leading term for large values of the parameters is used. 

The assessments are performed under assumptions allowing the application of all the methods explained above. In other words, we will evaluate the diffraction integrals $U=U(r, \phi; f)$, defined in \eqref{dif_int}, in a grid of points   (in polar   coordinates) $(r,\phi)\in\R^{N\times N}$, and for a vector $\bm f\in \R^M$ of values of the defocus parameters $f$.  Furthermore, for the ENZ formulas we take into consideration only the terms containing the cosine, so that the original ENZ-PB formulas can be used.

\subsection{Some remarks about FFT vs semi-analytic methods}

In the FFT2-based scheme, the value of $U(r,\phi; f)$ is computed by means of the bi-dimensional fast Fourier transform. The crucial step is the substitution of the double integral in \eqref{dif_int} by a discrete sum. Notice that each new value of the defocus parameter $f$ obliges to calculate the values of $U$ completely, at a  computationally high cost. Furthermore, the use of the FFT requires re-sampling the wavefront at a regular Cartesian grid covering the pupil; for convenience, the length of the grid  should be an integer power of 2 in each direction. 

Another remarkable issue to be addressed when using the FFT2 scheme is the aliasing, 
 a typical phenomenon that can appear due to discontinuities of the integrand. In order to prevent this, the pupil must be small in comparison with the sampled area (or in other words, we must extend the pupil to a larger region, setting the pupil function to zero in the complementary domain), resulting in a large area where $U(r,\phi; f)$ is negligible \cite{Gai2007}. Thus, a big portion of the computational load of this scheme is useless, and in general the spatial resolution needed with this method will be much higher than that required for the semi-analytic methods like discussed here. However, FFT2 is the standard method used in commercial ray tracing packages, such as Zemax or Code V.

The advantage of the semi-analytic approaches, such as the ENZ theory or the method proposed here, is that they reduce the computation of $U$ in \eqref{dif_int} to evaluation of more or less complex explicit expressions in terms of some elementary or special functions. One of the benefits of having these formulas is  a better control of the image domain  being computed, increasing the precision. Another important advantage is the huge boost in performance gained when a parallelization of calculations is done for multiple values of the defocus parameter $f$.
 
\subsection{Computational complexity}

Let us analyze and compare the computational cost of evaluating the diffraction integral  using the ENZ theory and by the GRBF scheme, for a grid of values $\bm r \in \R^N$, $\bm \phi \in \R^N$ (so that with the notation of Section~\ref{evaluation_grbf}, $J=N^2$), and for $M$ distinct defocuses, gathered in a vector $\bm f \in \R^M$.

\subsubsection*{The ENZ--PB method}

We start with the basic ENZ method, corresponding to formulas \eqref{expression:ENZ:Utotal}-\eqref{enz_coefs}.  
Observe that in~\eqref{B_k} we must compute values of the Bessel functions of the first kind $J_\nu$, usually by their power series
\begin{equation}
\label{bessel_def}
\mathop{J_{\nu}\/}\nolimits\!\left(z\right)=\left(\tfrac{1}{2}z \right)^{\nu}\sum_{k=0}^{%
\infty}(-1)^{k}\frac{(\tfrac{1}{4}z^{2})^{k}}{k!\mathop{\Gamma\/}\nolimits\!%
\left(\nu+k+1\right)}
\end{equation}
(see e.g.~[formula (10.2.2)]\cite{NIST}). In practice this  series must be truncated at some term $B-1$. For the values of variables and parameters usually appearing in this context, we have determined experimentally that we can  ensure the truncation error of  $10^{-8}$ taking $B=15$. The computational complexity of computing the Gamma function at a value of $\Ord{10^{t}}$ is $\Ord{ t^2 \log(t)^2 }$. With  $B=15$ and with the usual Zernike fit (up to the 8th order) it is sufficient to take $t\leq 3$, which yields a maximum of $11$ operations for the Gamma function evaluation. Thus, for $N$ different values of $r$ we need roughly $\Ord{ 3N+25 }$ operation to evaluate a single term in the series \eqref{bessel_def}, and in consequence, the computational cost of evaluating a single Bessel function is $\Ord{3BN}$. 

With the same assumptions, the evaluation of each coefficient in \eqref{enz_coefs} is at most $\Ord{150}$ flops. Consequently, the complexity of computation of each $B_k$ defined in \eqref{B_k} is $$\Ord{ (p+1)(30B+186) }$$ operations. 

In order to evaluate expression \eqref{Vnm}, we  truncate the infinite series at some term $S-1$, in which case the number of operations required to evaluate $V_n^m$ by \eqref{Vnm} is 
$$ \Ord{ S((p+1)(30B+186)+2)+S+8 } = \Ord{ 30 B S (p+1) }. $$
Propagating this estimate to \eqref{Unm}, computed also at $N$ different values of $\phi$, we need about $\Ord{  30 B S (p+1)+2 N^2 M+2N}$ flops for its evaluation. 

In order to estimate the overall cost for expression \eqref{expression:ENZ:Utotal} we have to take into account the indices $n$ and $m$ therein: $n$ ranges from $0$ to a certain $n^*$ (the maximum radial order), and for each fixed $n$, $m$ takes even/odd non-negative integer values $\leq n$. Thus, at each level $n$, only $\Ord{n/2}$ Zernike polynomials are used (recall the restriction of $m\geq 0$ for ENZ formulas). Then, $p=\frac{1}{2}(n-m) = \Ord{n/4}$ in average.

In addition, the maximum radial order $n^*$ yields a total of $\frac{1}{2}n^*(n^*+1)$ Zernike polynomials, from which only those with the cosine are  used, so that the effective number of polynomials is approximately $K=\frac{1}{4}n^*(n^*+1)$.

Combining all the previous assumptions, the complexity of computing $U$ at a grid of $N\times N \times M$ points for $(r,\phi,f)$ by the ENZ formulas can be estimated to be
$$ cost(U_{ENZ-PB}) = \Ord{ K^{3/2} BNS + K \left( N^2M + NMS + BNS \right) },  $$
where $B$ and $S$ are the truncation values for the series in \eqref{bessel_def} and \eqref{Vnm}, respectively, and $K$ is the total number of Zernike polynomials included in the pupil representation \eqref{expression:pupil:ENZ}. 

\subsubsection*{The ENZ--BB method}

We can perform a similar analysis for the  improved ENZ scheme, corresponding to equations \eqref{expression:ENZ:coeff1}-\eqref{expression:ENZ:coeff4}. The only difference with respect to the previous discussion lies in the computation of $V_n^m$, so we only need to re-estimate the computational cost of $V_n^m$ in equation \eqref{Vnm_improved}.

With the same assumptions used before, the evaluation cost of coefficients in equations  \eqref{expression:ENZ:coeff1}, \eqref{expression:ENZ:coeff2}, \eqref{expression:ENZ:coeff3} is of at most of $110$ flops. 
Thus, the complexity of coefficients $w_{k,l}$ in formula \eqref{expression:ENZ:coeff4} is of $\Ord{330 p k}$ flops. Recalling the estimate of $\Ord{3BN}$ flops for each Bessel function $J_{\nu}$, the cost of the finite sum (that with index $l$) in expression \eqref{Vnm_improved} is $ \Ord{(k+p) ( 3 BN + 330 p k + 3N)} $.

The cost of evaluation of the Bessel spherical function $j_k$ is $\Ord{ 3BM + 3M} = \Ord{ 3BM }$. Thus, if the improved formula for $V_n^m$ is used with the  truncation of the infinite series at the term $S-1$, its complexity is
$$ \Ord{S \left(  3BM + BN + (S/2+p)( 3BN+165 pS )  \right) + 2 M }. $$
Assuming as before that $p=\frac{1}{2}(n-m) = \Ord{n/4}$ in average, it can be simplified to
$$ \Ord{ nS^2 \left( 21S + 10 n \right) + 1.5 BNS^2 + \frac{3}{4} nS  B N + NMS + BMS }. $$

Since the rest of the calculations is the same as in the basic ENZ scheme,  the complexity of evaluating $U$ in a grid of $N\times N \times M$ points   $(r,\phi,f)$ by the improved ENZ formulas \eqref{expression:ENZ:coeff1}-\eqref{expression:ENZ:coeff4} is
\begin{align*}
 cost(U_{ENZ-BB}) = & \mathcal{O} \Bigl(  K^2 S^2 + K^{3/2} \left(S^3+BNS \right) +  \\
& + K \left(N^2M + NMS + BNS + S^3+S^2BN \right) +  N^2 M \Bigr),
\end{align*}

where $B$ and $S$ are the truncation values for the series in \eqref{bessel_def} and \eqref{Vnm}, respectively, and $K$ is the total number of Zernike polynomials included in the pupil representation \eqref{expression:pupil:ENZ}.

  \subsubsection*{The ENZ--EBB method} 
  
  This scheme is almost the same as in the ENZ--BB method, except for the rearrangement of the computation of the term $V_n^m$. Thus, only the complexity estimate of this term needs to be updated. 
  
  With the main basic assumptions made before, and since the Gamma function is involved in the Wigner 3j-symbols, the number of operations needed to compute each coefficient $A_{n_1,n_2,n_3}^{m_1,m_2,m_3}$ is of about 250. Recalling the estimate of $\Ord{3BN}$ flops for each Bessel function $J_{\nu}$, the cost of the innermost sum in expression \eqref{ENZ_EBB_Vnm} (the finite sum of index $h$) is $ \Ord{3 nBN + 250 n + 3N)} $.
  
  As in the previous case, the cost of evaluation of the Bessel spherical function $j_k$ is $\Ord{ 3BM }$ and then the complexity of expression \eqref{ENZ_EBB_Vnm}, when truncating at the term $S-1$, is approximately 
  $$ \Ord{S \left(  3BM + n( 3BN +  26 B + 8 N )  \right) + MN }. $$
  
  Combining the cost of $V_n^m$ with that of the rest of the ENZ BB scheme studied above, the complexity of evaluating $U$ in a grid of $N\times N \times M$ points   $(r,\phi,f)$ by the enhanced ENZ BB formulas \eqref{ENZ_EBB_Vnm}-\eqref{ENZ_EBB_A} is
  \begin{align*}
  cost(U_{ENZ-EBB}) = & \mathcal{O} \Bigl( K^{3/2} \left(BNS + NS + BS \right) + \\
  	& + K \left(N^2M + NMS + BNS + BMS \right) + N^2 M  \Bigr), 
  \end{align*}
  where $B$ and $S$ are the truncation values for the series in \eqref{bessel_def} and \eqref{Vnm}, respectively, and $K$ is the total number of Zernike polynomials included in the pupil representation \eqref{expression:pupil:ENZ}.

  \subsubsection*{The GRBF method}

Now we switch to the computational complexity of the GRBF based formulas described in Section~\ref{sec:calculation}.
Again, we are interested in evaluating expression \eqref{U1truncated} at a grid $\bm r \in \R^N$, $\bm \phi \in \R^N$ and $\bm f \in \R^M$, with the additional assumption \eqref{lambdaequal}, so that formula \eqref{U1truncated_withH} is used.

 According to the discussion in Section~\ref{evaluation_grbf}, once the cut-off parameter $S$ has been chosen, the double recurrence \eqref{rec_hatmk}-\eqref{rec_hatmkInit} must be evaluated for $s \leq S$, with an estimated computational cost of $\Ord{ 4 SM }$ flops. After that, the  calculation  of $\Omega$ in \eqref{omegaBis} takes about $\Ord{7N^2K}$ flops. The computational load of finding $\bm d$ by  \eqref{rowD} is almost negligible, requiring only $\Ord{4K}$ flops. Recall that  $\Omega$ and $\bm d$ are computed only once during the initialization of the algorithm. 

The computational complexity of evaluating each term $H_s$ in the mesh is of  $\Ord{ K+2N^2 }$ operations, yielding a total of $\Ord{ KS + 2SN^2 }$ flops for the evaluation of all the required $H_s$ terms. 

Summarizing, the cost of  calculating $U$ in \eqref{U1truncated}  according to Algorithm 1 is of  $$\Ord{2 N^2MS + 7 K N^2 + N^2M }.$$

Recall also that the minimal estimated computational complexity of the FFT2 method for a single value of $f$, even for optimal implementation, is of $\mathcal{O}(N^2 \log(N) )$, which corresponds to the cost of the FFT, the most computationally demanding part.

Table \ref{table1} summarizes the leading terms of the computational cost of each method, with the assumption that the same values of $B$ and $S$ for all semi-analytical methods have been chosen. A quick comparison shows that the GRBF approach is much more efficient than the ENZ theory, especially for a large amount of values for $f$. The FFT2-based method seems to be of similar complexity to GRBF, but in practice the number of sample points $N$ needed to achieve a reasonable accuracy for FFT2 will be much larger than that required for GRBF, which in our experiments was set to $400$ (see the description in the next section).  

\begin{table} 
\centering \begin{tabular}{c|c|c}
    Method & Complexity (single $f$ )& Complexity (vector of $f$) \\
\hline
    FFT2   	& $ \mathcal{O}(N^2 \log(N) )$     	& $ \mathcal{O}(M N^2 \log(N)   )$ \\
    ENZ--PB    	& $\mathcal{O}( N^2 K + N K^{3/2} )$  	& $\mathcal{O}( N^2 M + N^2 K M + N K^{3/2} )$   \\
    EN--BB & $\mathcal{O}(  N^2 K + N K^{3/2} + K^2 )$  	& $\mathcal{O}( N^2 M + N^2 K M + N K^{3/2} + K^2 )$   \\
ENZ--EBB &   $\mathcal{O}(  N^2 K + N K^{3/2} + K^{3/2} )$  &   $\mathcal{O}( N^2 M + N^2 K M + N K^{3/2} + K^{3/2} )$ \\
    GRBF   	& $\mathcal{O}( N^2 K)$     			& $\mathcal{O}( N^2 M + N^2 K )$ \\
\end{tabular}
\caption{Estimates of the minimal computational complexity of the methods when evaluating $U$ at a $N \times N$ grid of nodes $(r,\phi)$ and for $M$ values for the defocus parameter $f$, using $K$ functions in the corresponding series expansions (for ENZ and GRBF).} \label{table1}
\end{table}

\subsection{Execution time} 

Since the complexity estimates give only a rough idea  of the computational demand of a method, we have looked also at the execution time, which  is a simpler and a more informative assessment.  
 
For the beginning, we run the ENZ-based and the GRBF algorithms evaluating $U$ at an $100 \times 100$ mesh of nodes for a single value of the parameter $f$, recording the execution time in dependence of the number of functions used in the corresponding series expansions.  The experiments were performed on a standard PC running Matlab v.~8.1. Figure \ref{figure:times} shows the results, along with the corresponding regression curves.
The slope of the regression line is approximately $0.0612$ seconds/function for ENZ-PB, $0.0144$ for ENZ-EBB and $0.0028$ seconds/function for GRBF; although the dependence for ENZ-BB is clearly non-linear, the average slope of the regression curve in the studied range for this case is $0.2151$. Thus, the GRBF approach is at least $5$ times faster in comparison with its closest competitor, the ENZ-EBB method. 
For comparative purposes, the execution time to evaluate function $U$ numerically making use of the two-dimensional fast Fourier transform was of approximately $0.25$ seconds, matching the execution time for ENZ using about $12$ Zernike terms, or for GRBF with approximately $90$ Gaussian RBFs.

\begin{figure} 
\centering \includegraphics[width=0.85 \linewidth]{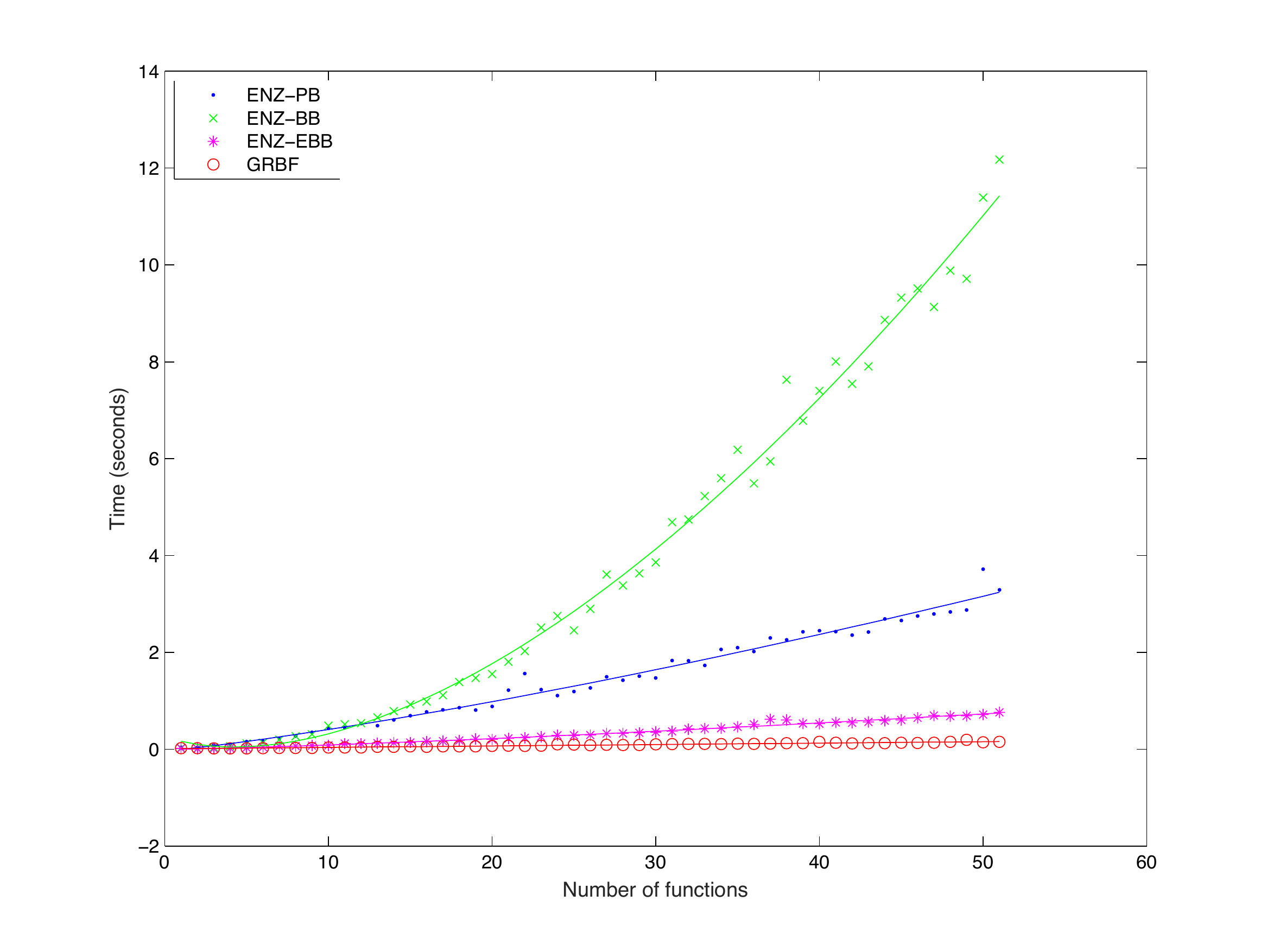}      
 \caption {Dependence of the computation time from the number of functions used in the description of the complex pupil function, by means of the ENZ--PB (formulas \eqref{Vnm}--\eqref{B_k}), the ENZ--BB (formulas \eqref{Vnm_improved}--\eqref{expression:ENZ:coeff3}), the ENZ--EBB (formulas \eqref{ENZ_EBB_Vnm}--\eqref{ENZ_EBB_A}) and the GRBF methods.}
\label{figure:times}
\end{figure}

In order to achieve a reasonable comparison, in the rest of our experiments we implemented a standard setting for each of the methods. Namely, the integral $U$ was evaluated at a $100 \times 100$ regular grid of nodes by the semi-analytic methods, and at a $512 \times 512$ regular grid for the FFT2-approach (this is a realistic size needed to overcome aliasing and obtain accurate results). Moreover, we used $K=K_{GRBF}=400$ functions in the formula~\eqref{U1truncated} (corresponding to the $20\times 20$ regular grid of centers, as described above) and the cut-off parameter $S=60$, while for the ENZ methods it was observed that a total of $K=K_{ENZ}=45$ Zernike polynomials (up to the 8th order polynomials) was the optimal. This is actually about twice the number of terms recommended by \cite{Braat_etal2002,Janssen2002}. Higher order polynomials, according to our experiments, did not contribute to higher accuracy, while causing ill conditioning of the computations.  

With these settings, we also compared the execution time of these methods as a function of the length of the vector of defocus parameters $\bm f$. The results appear in Figure \ref{figure:times2}, along with the regression lines for each scheme. The values of the slopes are approximately $0.139$ for ENZ-PB, $0.064$ for ENZ-BB, and for ENZ-EBB and $0.004$ for GRBF, all in seconds per size of $\bm f$. This means that ENZ-PB is about $35$ times slower than GRBF when calculating $U$ for many of values of f simultaneously, while ENZ-BB and EBB are about $16$ times slower than GRBF, even though the number of Gaussian functions used was much higher than the number of Zernike polynomials.

\begin{figure} 
\centering \includegraphics[width=0.85 \linewidth]{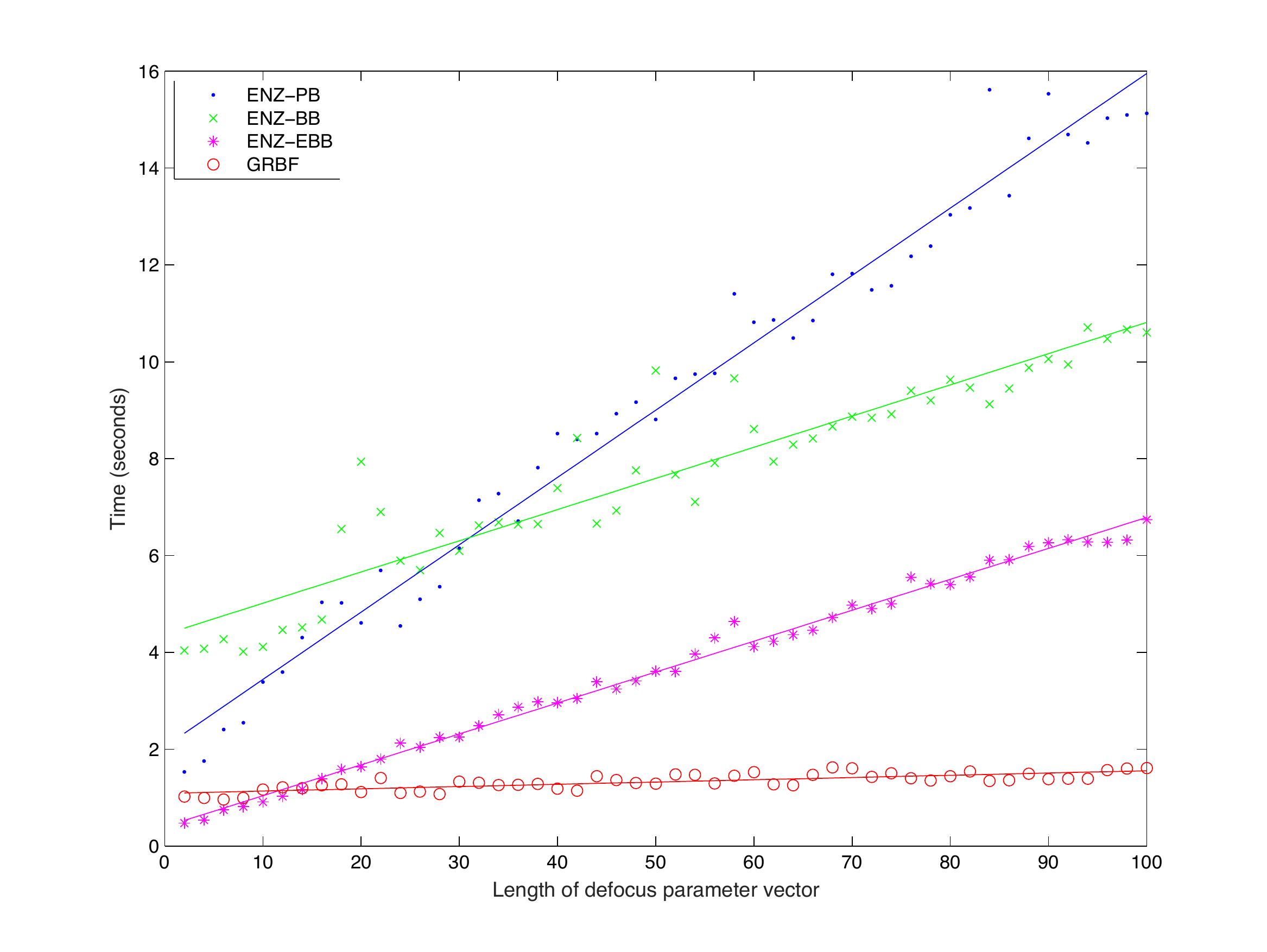}      
 \caption {Execution time according to the number of different defocus parameters used in the complex pupil function. A fixed number of functions ($ 400$ for GRBF and $ 45$ for ENZ) was used for each value of $f$. } 
\label{figure:times2}
\end{figure}

\subsection{Accuracy}

The original ENZ theory, although representing a big step forward, had some limitations that had to be addressed. The obvious one is the use of only even terms in the Zernike expansion of the complex pupil function, which restricts it to the symmetric wavefront errors. Some other, less evident, problems lie in the core of the mathematical properties of the ENZ explicit formula for $U(r,\phi; f)$. This is an infinite series of terms, each of them a finite linear combination of Bessel functions, and each new  Zernike term added to the expansion of the pupil function \eqref{pupil} increases the complexity of the terms. The series is slowly convergent, especially for larger values of $f$, requiring a truncation with a large number of terms  depending on  $f$  (it is recommended to use $3|f|+5$ terms, according to \cite{Braat_etal2002,Janssen2002}).

Another issue in evaluating the ENZ expressions is the accuracy. The terms of the infinite series with even and odd orders form sign-changing sequences, which increases the risk of the cancellation errors. This phenomenon can be illustrated by the following experiments: in the case when the wavefront is given only by a positive $Z_2^2$ horizontal astigmatism, the evaluation of, say, the imaginary part of $U$ at a point with radial coordinate $r=0.9$ consists in adding a finite alternating sequence, with two dominant terms of approximately $0.423$ each, with opposite signs, but whose absolute values differ in $3\times 10^{-5}$. This shows that these calculations, if not well organized, can yield a loss of precision in about 5 significant digits. Last but not least, the ENZ formulas contain binomial numbers that must be evaluated with care in order to avoid overflow.

Although the enhanced scheme of evaluation of the ENZ formulas \eqref{ENZ_EBB_Vnm}--\eqref{ENZ_EBB_A} represents a significant improvement in speed, as the experiments show, still the danger of the cancellation errors is looming in the horizon due to the alternating sign terms in \eqref{ENZ_EBB_Vnm}.

In comparison, the error estimates from Section~\ref{sec:error estimates} show that the convergence in \eqref{U1} is very fast, and only a very reasonable number of terms is required for a precise evaluation.

In order to assess accuracy, we carried out a number of numerical experiments where we calculated $U$. For the ideal wavefront ($\Phi\equiv 0$) and zero defocus ($f=0$), the results were discussed in \cite{RMI2014}, where it was shown that  the two semi-analytic methods performed similarly, although GRBF calculations were done at a much lower cost.

More informative is to use a synthetic wavefront  described by a combination of Zernike polynomial terms and exponential function, seeking a ``fair'' comparison between the ENZ (``Zernike--oriented'') and the GRBF (``exponentially--oriented'') approaches. Such class of functions allows also for an easy modeling of low and high-order oscillations. For experiments with a wavefront comprised of basically low-frequency oscillations see~\cite{RMI2014}. Here we have used the function given by
\begin{equation} \label{wavefrontDef}
\begin{split}
\Phi(\rho, \theta) &=   0.6 \, Z_5^3(\rho, \theta) - 0.4 \, Z_4^4(\rho, \theta)  -    0.3\,  Z_5^5 ( \rho, \theta)   + 0.25 \, Z_4^2(\rho, \theta) \\ & +0.25 \, Z_6^4(\rho, \theta)    -    0.15\,  Z_8^4 ( \rho, \theta)  +      0.4\, g(\rho \cos \theta ,\rho \sin \theta ; -0.3,0,15) \\ & -      2\, \left[ g(\rho \cos \theta ,\rho \sin \theta; 0.5,0.3,10)    + g(\rho \cos \theta,\rho \sin \theta; 0.5,-0.3,10) \right],
\end{split}
\end{equation}
where $Z_n^m$ are the (orthonormal) Zernike polynomials and
$$
g(x,y; a, b, \lambda) = \exp{ \{  - \lambda [ ( x - a )^2 + ( y - b )^2 ] \} },
$$
as well as its ``contaminated'' version, where we have added to $\Phi$ a normally distributed random noise of the form \texttt{0.5 * randn( )} at a grid of $100 \times 100$ equally spaced points (see Figure \ref{figure:wavefront}).

\begin{figure}
\centering \includegraphics[scale=0.4]{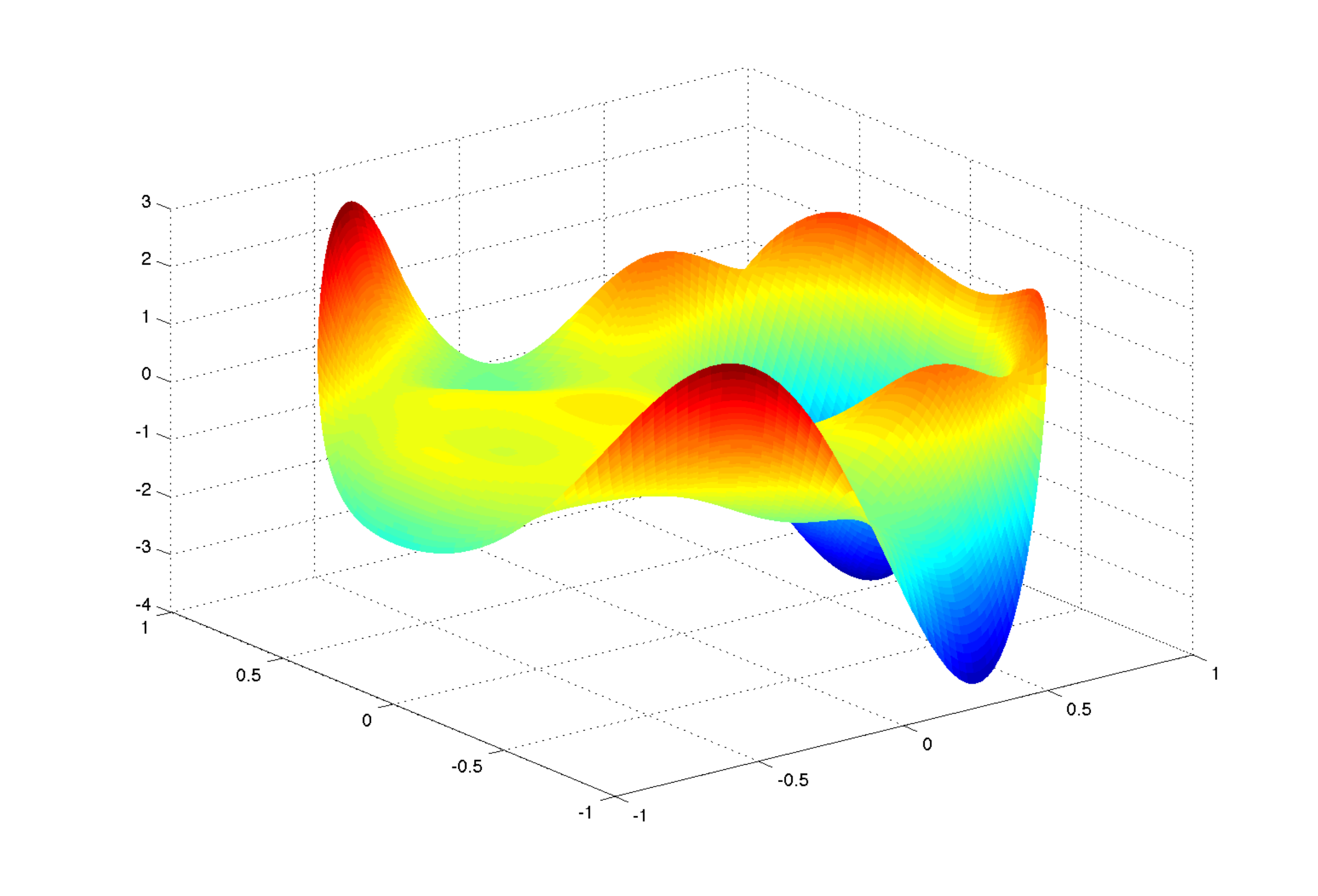} \\
\includegraphics[scale=0.4]{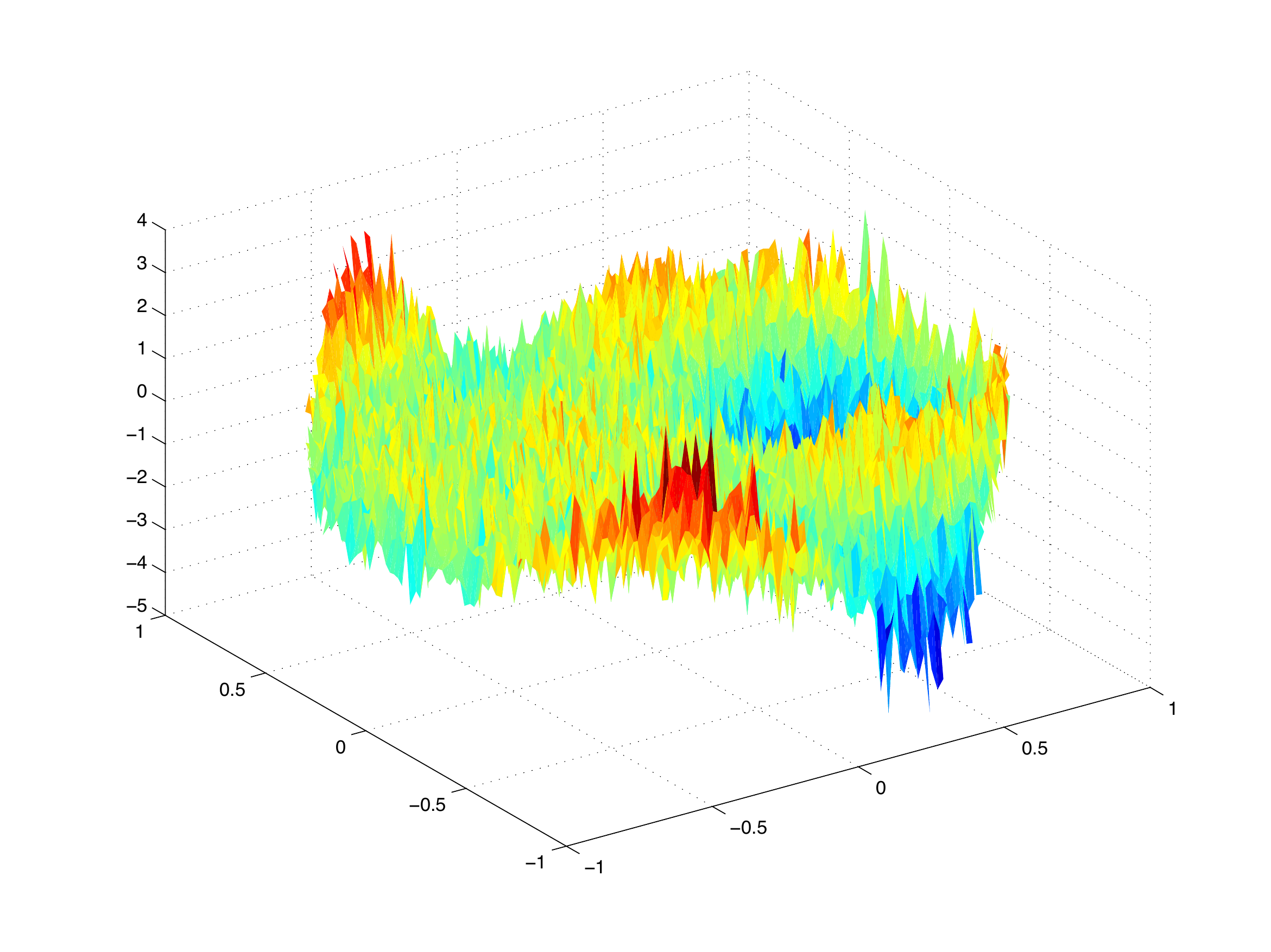}  
 \caption {3D plot of the synthetic wavefront function defined in \eqref{wavefrontDef} (upper picture) and of the same wavefront contaminated by a white noise $0.5 \times  N(0,1)$ (bottom).}
\label{figure:wavefront}
\end{figure}
For these wavefronts we calculated the diffraction integral \eqref{dif_int} for different values of $f$ by quadrature using the scientific software \emph{Mathematica} with extended precision (with the options \texttt{PrecisionGoal} set to 8 and \texttt{WorkingPrecision} to 16). These values of $U$ (regarded as ``exact'') were compared with the calculations performed by FFT2 and by two semi-analytic approaches discussed here, for which the diffraction integral was evaluated by all methods in a grid of $256 \times 256$ equally spaced points in the square $[-2,2] \times [-2,2]$.

Seeking a fair comparison of the computation of the integrals by both semi-analytic methods we 
fitted the pupil function using the first 200 Zernike polynomials for the ENZ, while for the GRBF the approximation was performed by the linear combination of $20 \times 20$ Gaussian functions, with the parameter $\lambda=16$.  With this choice of the number of functions in each basis an approximately similar  goodness-of-fit to the wavefront data with noise is reached (the RMS error aproximately $0.09$ for both methods). 

As an illustration of comparative performance of the computation methods for different values of the defocus parameter $f$ we plot in Figure~\ref{figure:PSFseveralF} the values of the normalized PSF, again  for the simulated wavefront \eqref{wavefrontDef}, along the horizontal line ($\phi=0, \pi$ and $r\in [0,1]$), setting $f=0$, $f=2\pi$ and $f=-2\pi$, as in a numerical experiment described in \cite{Braat_etal2002}. In the case $f=0$ the dotted line (computed by the ENZ method) is not observed because it matches exactly the other two curves, found using quadrature and the GRBF algorithm. We see that  for larger values of defocus our procedure outperforms the alternative methods.

\begin{figure}
\centering 
\hspace{-3mm} \begin{tabular}{cc}
\includegraphics[width=0.47 \linewidth]{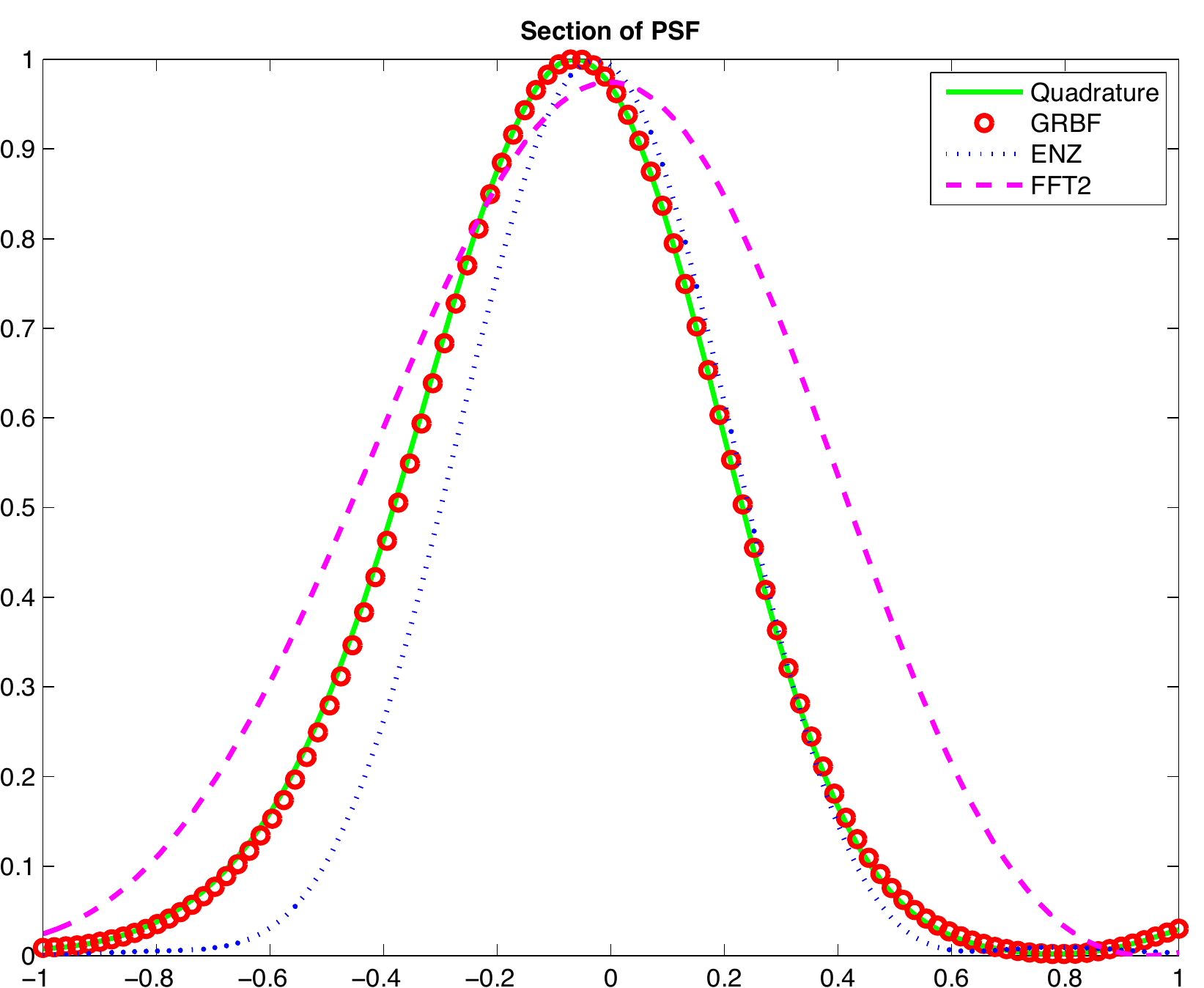} & \includegraphics[width=0.47 \linewidth]{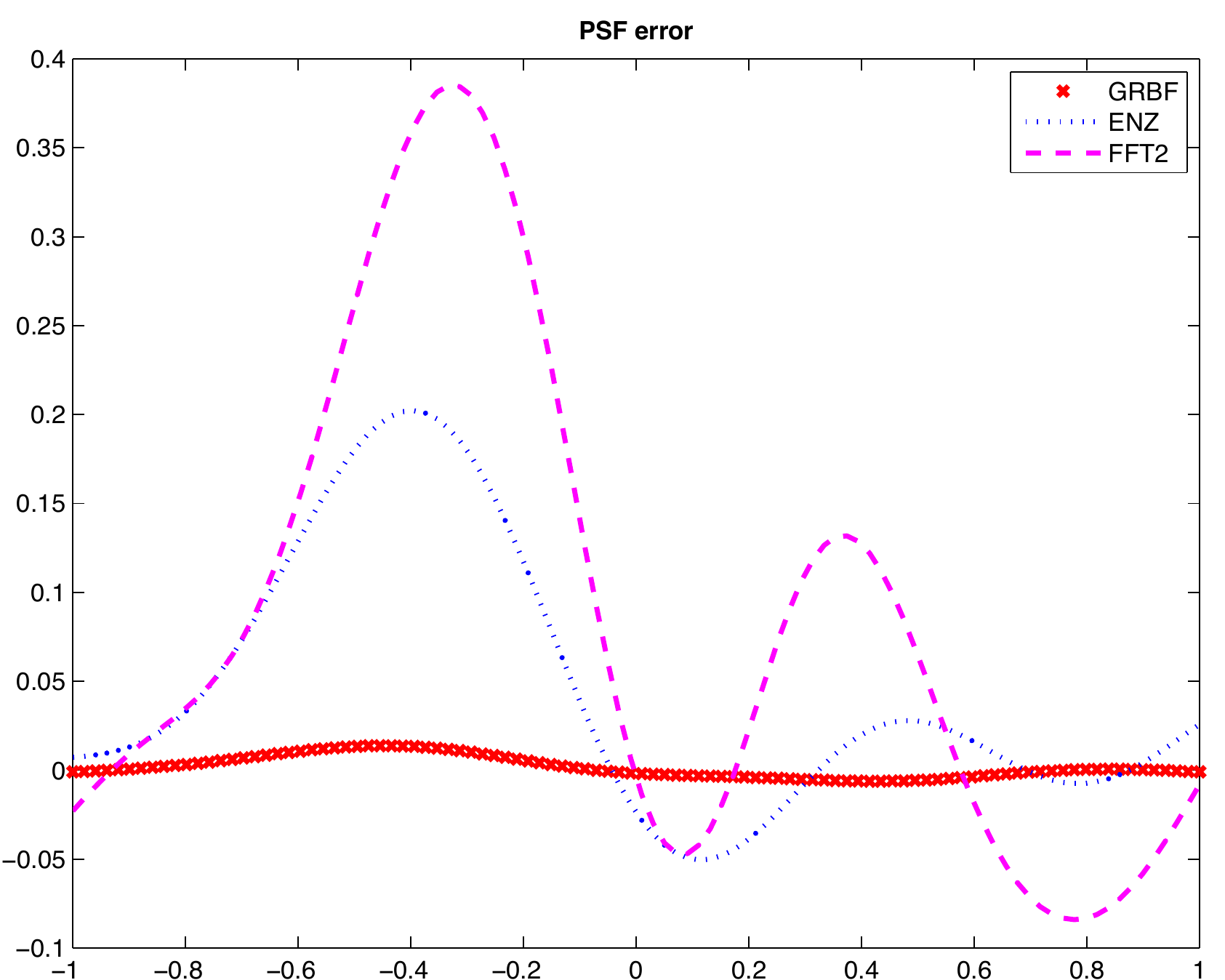} \\
\includegraphics[width=0.47 \linewidth]{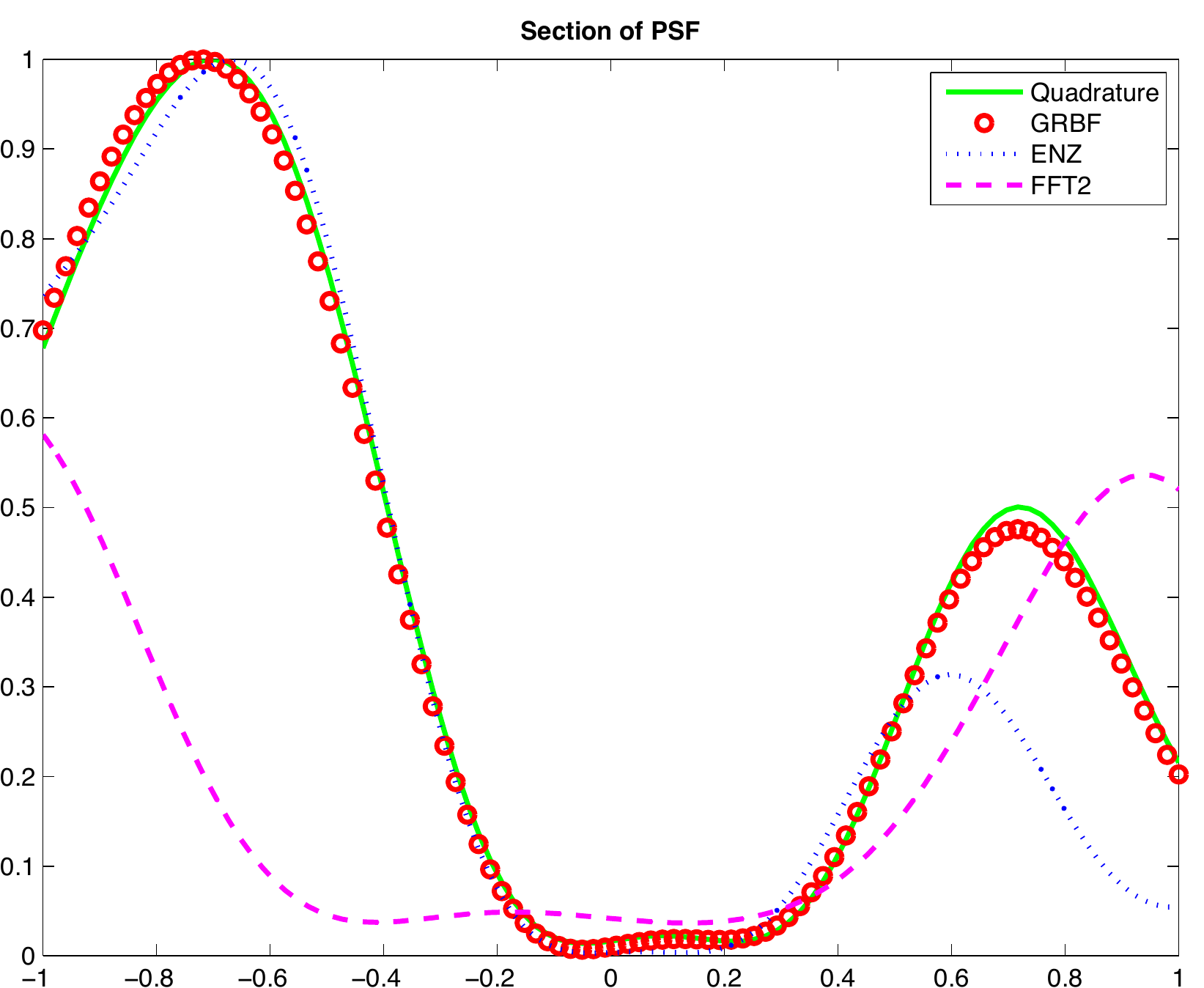} & \includegraphics[width=0.47 \linewidth]{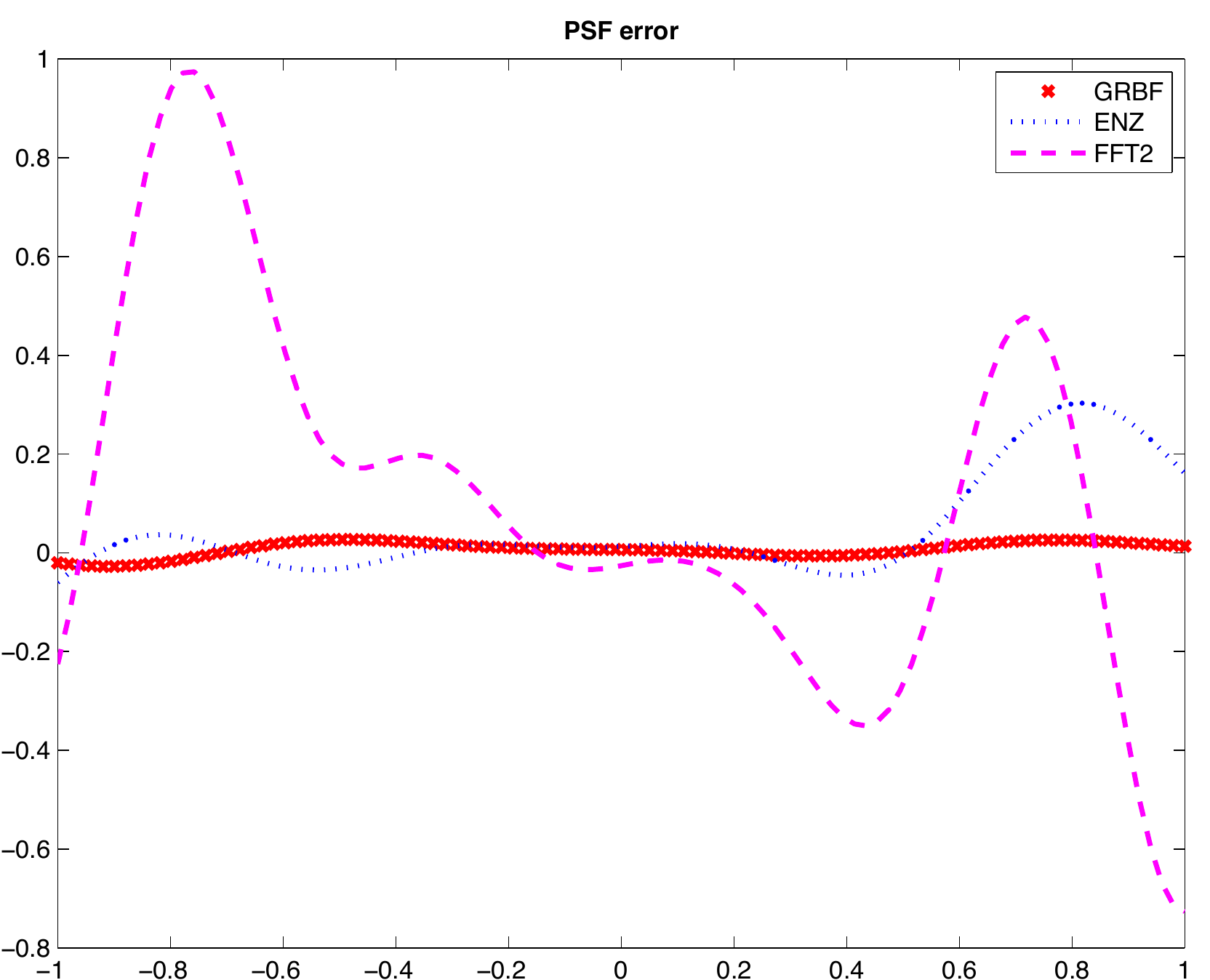} \\
\includegraphics[width=0.47 \linewidth]{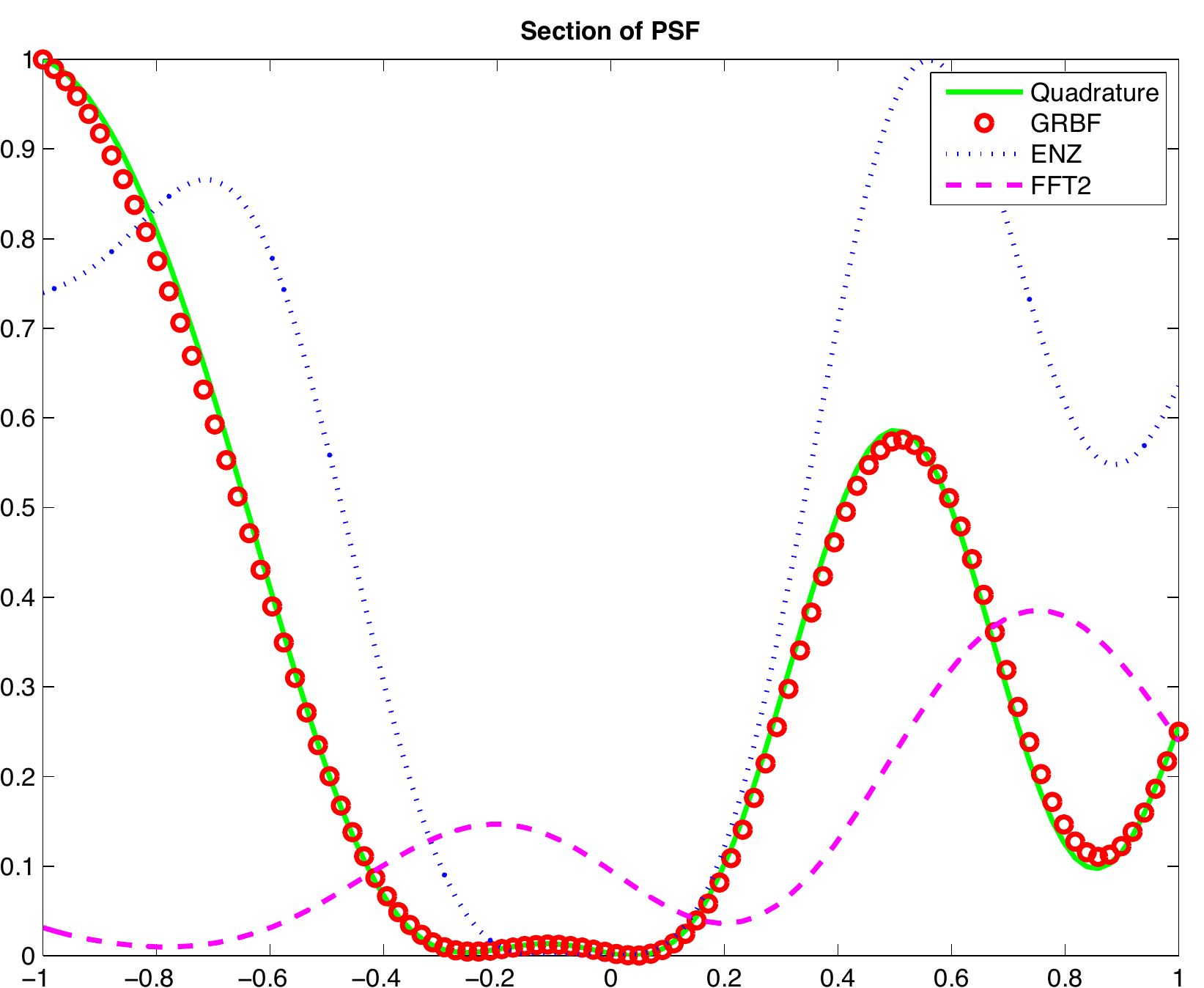}  & \includegraphics[width=0.47 \linewidth]{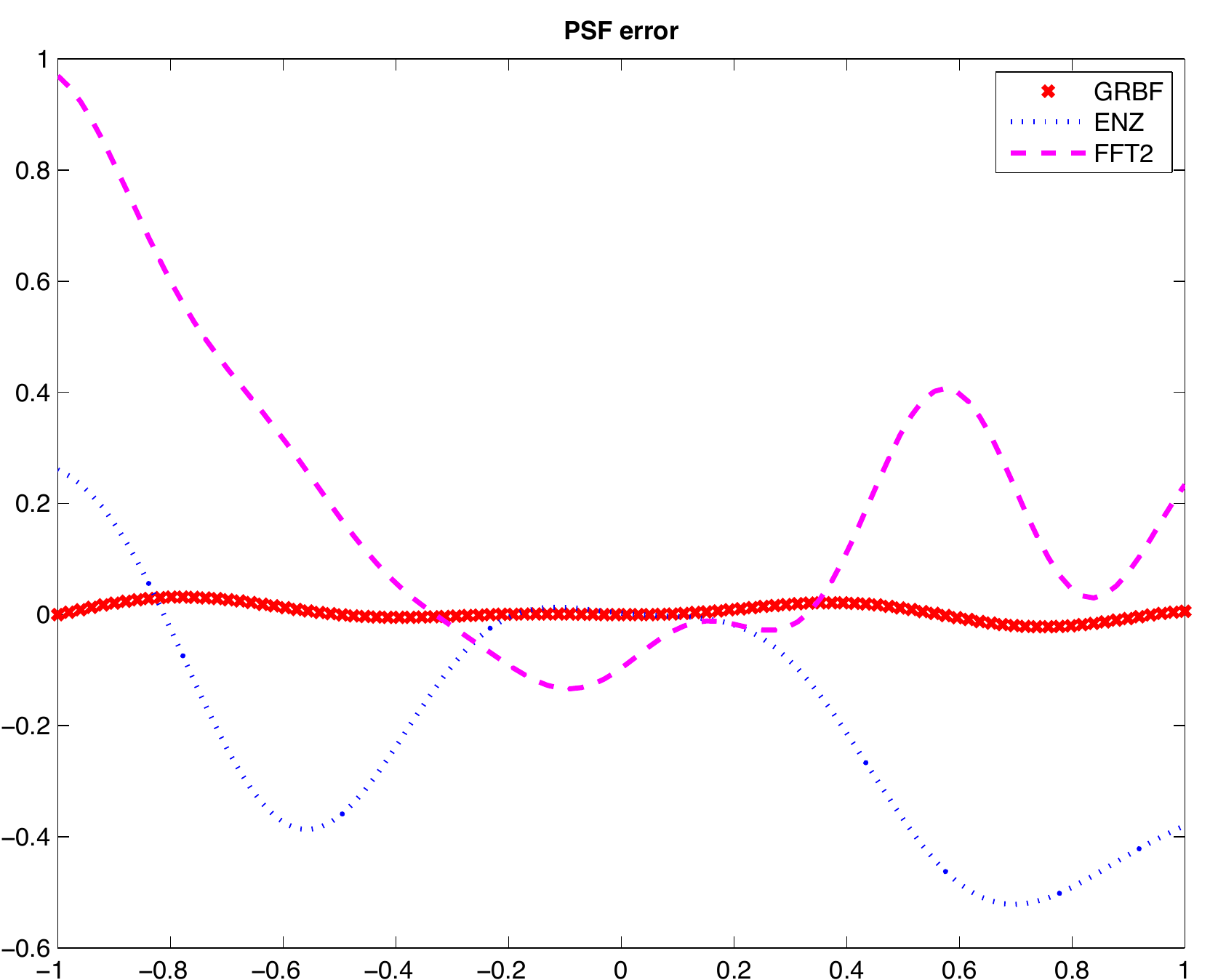}
\end{tabular}
 \caption {Values of the normalized PSFs and the corresponding errors (difference between the PSF computed by quadrature and by each other method, as explained in the text) for the wavefront~\eqref{wavefrontDef} $+$ noise, along the horizontal diameter of the unit disk, calculated for each method, for $f=0$ (top row), $f=2\pi$ (middle) and $f=-2\pi$ (bottom row).}
\label{figure:PSFseveralF}
\end{figure}

In our last experiments we analyzed the performance of all methods for a non-circular geometry. Namely, for the same synthetic wavefront as before we set an elliptic pupil, taking in~\eqref{pupil} as $A(\rho, \theta)$ the characteristic function of the set (in cartesian coordinates)
$$
x^2+\frac{y^2}{(0.7)^2}\leq 1.
$$
The results are illustrated in Figure~\ref{figure:PSFelliptic}. The poorer performance of the ENZ method in these examples is due probably to a less accurate fit of the complex wavefront by Zernike polynomials over a non-circular domain.

\begin{figure}
\centering 
\begin{tabular}{lr}
  \hspace{-5mm} \includegraphics[width=0.5 \linewidth]{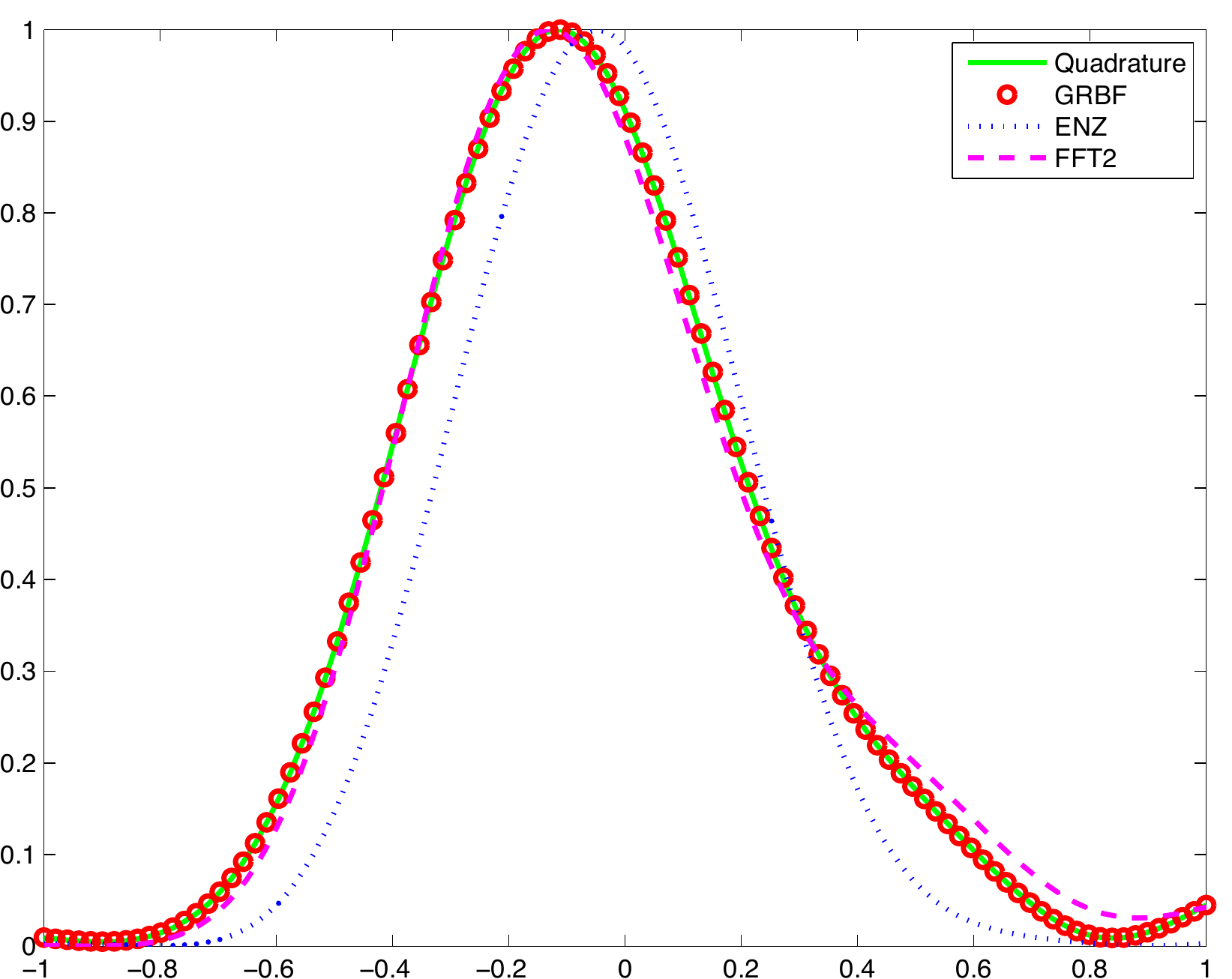} & \hspace{-3mm}
\includegraphics[width=0.5 \linewidth]{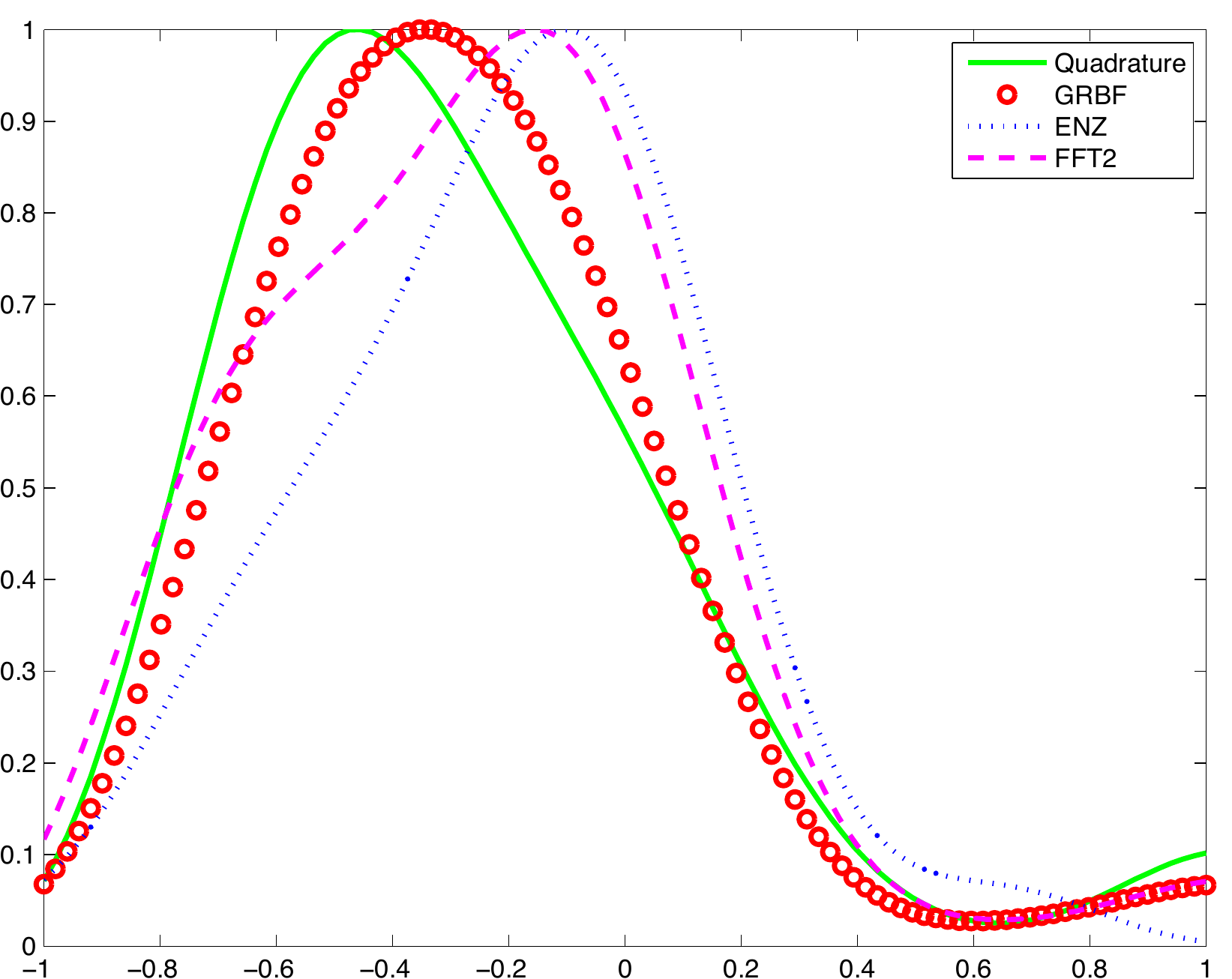}
\end{tabular}
 \caption {Values of the normalized PSFs for the  synthetic wavefront  with an elliptic pupil along the horizontal diameter of the unit disk, calculated for each method, for $f=0$ (left)  and $f=\pi$ (right).}
\label{figure:PSFelliptic}
\end{figure}

\section{Conclusions}

A new procedure for computing 2D Fourier-type integrals, and in particular, the diffraction integrals with variable defocus has been put forward. We have discussed some error bounds and developed an efficient scheme for parallel evaluation of these integrals in a grid of points and for a vector of defocus parameters.
 
It should be noted that when calculating the Fourier-type integrals considered in this paper by replacing the original integrand by its approximant, special care should be put in the quality of approximation. It is easily seen that proximity in an $L^2$ or even uniform norm is not enough to guarantee small errors (due to a high sensitivity of these integrals to oscillations), and Sobolev-type norms should be used. The higher accuracy we achieve in the approximation step, the better results will be obtained for the integral calculation. In our case, a linear least squares fit with Tikhonov regularization gave the best results.
 
It is worth pointing out that the semi-analytic paradigm for computed Fourier-type integrals can be used with other bases, that might give other computational advantages, such as Sobolev orthogonal polynomials or polynomials introduced in \cite{Forbes}, as well as for the computation with the more general focal factors, like in \eqref{defocussing}.

The proposed here GRBF-based approach has been compared with the two existing procedures, i.e., the 2-dimensional fast Fourier transform and the extended Nijboer-Zernike theory. The results of the comparison show that the new scheme is very competitive, providing higher accuracy and speed. Among other advantages of our method we can mention the following:
\begin{itemize}
\item a relative robustness of the computation with respect to the underlying geometry of the pupil function. The GRBF have an almost local character (especially, for higher values of the shape parameter), and the fit that the linear combination of such functions provides is less affected by the shape of the pupil;
\item the existence of the shape parameter in the GRBF provides more flexibility to small details, and allows for an easy implementation of a multi-resolution scheme, especially in the case of existence of subdomains with different complexity of the integrand. Since each function used for approximation of the pupil function enters the final expression linearly, one can use two or more layers of GRBF to fit the residual error consecutively using different sets of centers and different shape parameters in order to improve the accuracy of the results; 
\item in the case of the calculation of diffraction integrals of the form \eqref{dif_int}, the increase of the computational cost  for a vector of values of the defocus parameter is practically negligible, providing a substantial increase in the performance with respect to the other techniques. This is a reliable and efficient way of obtaining the through-focus characteristics of an optical system at higher resolutions in reasonable time. 
\end{itemize} 
 
In conclusion, the proposed GRBF approach allows calculating through-focus point spread function at a very low computational cost for an arbitrarily selected set of the defocus parameters. This is particularly attractive in those applications in which evaluation of through-focus characteristics of an optical system is required. They include wavefront sensing, phase retrieval, lithography, microscopy, extreme ultraviolet light optics, digital holography and physiological optics. We believe that the GRBF-based method has an even wider scope of application in mathematical imaging and vision.

\section*{Acknowledgements}
 
 The first and the second authors (AMF and DRL) have been supported in part by the Spanish Government together with the European Regional Development Fund (ERDF) under grants
 MTM2011-28952-C02-01 (from MICINN) and MTM2014-53963-P (from MINECO), by Junta de Andaluc\'{\i}a (the Excellence Grant P11-FQM-7276 and the research group FQM-229), and by Campus 
 de Excelencia Internacional del Mar (CEIMAR) of the University of Almer\'{\i}a.  Additionally,  DRL was supported by the FPU program of the Ministry of Education of Spain.

This work was completed during a stay of AMF as a Visiting Chair Professor at the Department of Mathematics of the Shanghai Jiao Tong University (SJTU), China. He acknowledges the hospitality of the hosting department and of the SJTU.

Finally, the authors are grateful to S.~van Haver, who drew our attention to the most recent developments in the ENZ theory, and to the anonymous referees whose remarks helped to improve the original version of the manuscript.

\section*{References}


\end{document}